\documentclass[reqno]{amsart}
\numberwithin{equation}{section}

\usepackage[colorlinks,linkcolor={blue}]{hyperref}

\usepackage{bm}
\usepackage{amssymb}
\usepackage{graphicx} 
\usepackage{amscd}
\usepackage{enumerate}
\usepackage{cancel}
\usepackage[font=footnotesize]{caption}
\usepackage{dsfont}
\usepackage[colorinlistoftodos]{todonotes}
\usepackage{tikz-cd}

\definecolor{darkred}{rgb}{1,0,0} 
\definecolor{darkgreen}{rgb}{0,0.8,0}
\definecolor{darkblue}{rgb}{0,0,1}

\hypersetup{colorlinks,
linkcolor=darkblue,
filecolor=darkgreen,
urlcolor=darkred,
citecolor=darkgreen}

\makeatletter
\providecommand\@dotsep{5}
\makeatother

\definecolor{orange}{RGB}{253,85,0}
\definecolor{darkgreen}{RGB}{0,95,10}

\newcommand{\gray}[1]{{\textcolor{gray}{#1}}}
\newcommand{\dgray}[1]{{\textcolor{darkgray}{#1}}}

 \newcommand{\shift}{\mathrm{shift}}
 \newcommand{\N}{\mathds{N}}
 \newcommand{\Z}{\mathds{Z}}
 \newcommand{\R}{\mathds{R}}
 \newcommand{\C}{\mathds{C}}
 \newcommand{\lu}{\ell^{u}}
 \newcommand{\ls}{\ell^{s}}
 \newcommand{\F}{\mathcal{F}}
 \newcommand{\Wu}{W^{u}}
 \newcommand{\Ws}{W^{s}}

 \newcommand{\Eu}{E^{u}}
 \newcommand{\Es}{E^{s}}
 \newcommand{\Kbr}{K_{\mathrm{br}}}
 \newcommand{\Krad}{K_{\mathrm{rad}}}

 \newcommand{\FF}{\mathcal{F}}

 \newcommand{\GG}{\mathcal{G}}

 \newcommand{\RR}{\mathcal{R}}

 \newcommand{\area}{\mathrm{area}}

 \DeclareMathOperator{\interior}{int}

 \DeclareMathOperator*{\Fix}{Fix}
 \DeclareMathOperator*{\toup}{\longrightarrow}

 \theoremstyle{plain}
 \newtheorem{MainThm}{Theorem}
 
 \newtheorem{Thm}{Theorem}[section]
 
 \newtheorem{Lemma}[Thm]{Lemma}
 \newtheorem{Cor}[Thm]{Corollary}
 
 \theoremstyle{definition}

\title[Existence of Birkhoff sections]{Existence of Birkhoff sections for Kupka-Smale Reeb flows of closed contact 3-manifolds} 

\author{Gonzalo Contreras}
\address{Gonzalo Contreras\newline\indent 
Centro de Investigaci\'on en Matem\'aticas\newline\indent 
A.P. 402, 36.000, Guanajuato, GTO, Mexico}
\email{gonzalo@cimat.mx}

\thanks{Gonzalo Contreras is partially supported by CONACYT, Mexico, grant A1-S-10145. Marco Mazzucchelli is partially supported by the SFB/TRR 191 ``Symplectic Structures in Geometry, Algebra and Dynamics'',
funded by the Deutsche Forschungsgemeinschaft.}

\author{Marco Mazzucchelli}
\address{Marco Mazzucchelli\newline\indent CNRS, UMPA, \'Ecole Normale Sup\'erieure de Lyon\newline\indent 46 all\'ee d'Italie, 69364 Lyon, France}
\email{marco.mazzucchelli@ens-lyon.fr}

\date{November 11, 2021}

\keywords{Reeb flows, geodesic flows, surfaces of section, Birkhoff sections}

\subjclass[2010]{53D10, 37D40, 53C22}

\begin{document}

\begin{abstract}
A Reeb vector field satisfies the Kupka-Smale condition when all its closed orbits are non-degenerate, and the stable and unstable manifolds of its hyperbolic closed orbits intersect transversely.
We show that, on a closed 3-manifold, any Reeb vector field satisfying the Kupka-Smale condition admits a Birkhoff section. In particular, this implies that the Reeb vector field of a $C^\infty$-generic contact form on a closed 3-manifold admits a Birkhoff section, and that the geodesic vector field of a $C^\infty$-generic Riemannian metric on a closed surface admits a Birkhoff section.

\tableofcontents
\end{abstract}

\maketitle

\section{Introduction}
\label{s:intro}

Surfaces of section are fundamental tools that allow to reduce the study of the dynamics of a vector field $X$ on a 3-dimensional closed manifold $N$ to the study of the dynamics of a surface diffeomorphism. Formally, they are immersed compact surfaces $\Sigma\looparrowright N$ whose interior $\interior(\Sigma)$ is embedded and transverse to $X$, and whose boundary $\partial\Sigma$ is tangent to $X$ (that is, $\partial\Sigma$ is the covering map of a finite collection of closed orbits of $X$). In order to carry out the above mentioned reduction without loosing any information on the dynamics, the following extra property should hold: if $\phi_t:N\to N$ denotes the flow of $X$, for some $T>0$, any flow segment $\phi_{[0,T]}(z)$ must intersect $\Sigma$. Surfaces of section satisfying this extra property are called \emph{Birkhoff sections}. This notion originates from the seminal work of
Poincar\'e (see Conley~\cite{Conley63}) and  Birkhoff \cite{Birk3, Birkhoff:1966uu} in 
celestial mechanics.
In particular, Birkhoff showed that any simple closed geodesic of a positively curved Riemannian 2-sphere produces a surface of section (indeed,  an embedded one) diffeomorphic to an annulus for its geodesic vector field. 
This annulus-like surface of section was a crucial ingredient for the proof of the existence of infinitely many closed geodesics on any Riemannian 2-sphere, a celebrated statement that follows from the combination of the work of Bangert~\cite{Bangert:1993ue} and
Franks~\cite{Franks:1992wu}.
By a result of Fried \cite{Fried:1983uj}, any transitive Anosov vector field on a closed 3-manifold admits a Birkhoff section. In the same paper, Fried also showed a construction due to Birkhoff of a surface
of section of genus one for the geodesic flow of negatively curved closed Riemannian surfaces.

In symplectic dynamics, the quest of Birkhoff sections attracted plenty of interest in the last few decades. In their celebrated paper \cite{Hofer:1998vy}, Hofer, Wysocki, and Zehnder showed that the canonical Reeb flow on any convex 3-sphere embedded in $\R^4$ admits an embedded Birkhoff section diffeomorphic to a disk.  
An outstanding application of this result, combined with a result of Franks \cite{Franks:1992wu} on area-preserving surface homeomorphisms, is that such Reeb flows must have either exactly two or infinitely many closed Reeb orbits. 
In another celebrated paper \cite{Hofer:2003wf}, Hofer, Wysocki and Zehnder proved that any Kupka-Smale tight contact form in the 3-sphere admits a finite energy foliation, which is a generalization of the notion of open book decomposition. When such a finite energy foliation is not an ordinary open book, they showed the existence of homoclinics to certain binding closed Reeb orbits, which implies the exponential growth of closed orbits and even the positivity of the topological entropy of the Reeb flow.

While we were completing the present paper, in a talk at the Symplectic Geometry Zoominar \cite{Hryniewicz:2021ua}, Hryniewicz announced the following existence result for Birkhoff sections, joint with Colin, Dehornoy, and Rechtman.

\begin{Thm}[Colin, Dehornoy, Hryniewicz, Rechtman]
\label{t:CDHR}
On a closed 3-manifold $N$, a $C^2$-generic contact form admits a Birkhoff section for its Reeb vector field. If $N$ is a homology 3-sphere, the assertion holds for a $C^\infty$ generic contact form.
\end{Thm}

The main ingredients for the proof of this theorem are the  broken book decompositions. This is a generalization of the notion of rational open book decomposition, inspired by Hofer, Wysocki and Zehnder's finite energy foliations, and recently introduced by Colin, Dehornoy, and Rechtman \cite{Colin:2020tl}. The existence of such broken book decompositions for Reeb flows with non-degenerate closed Reeb orbits is based on techniques from embedded contact homology \cite{Hutchings:2014vp}. Other ingredients for Theorem~\ref{t:CDHR}, as sketched during Hryniewicz's talk, are the asymptotic cycles of Schwartzman, Fried, and Sullivan \cite{Schwartzman:1957ut, Sullivan:1976uq, Fried:1982uh}, Irie's equidistribution theorem \cite{Irie:2018wm}, and an action-linking relation due to Bechara Senior, Hryniewicz, and Salom\~ao \cite{Bechara-Senior:2021tq}.

In this paper, we establish the existence of Birkhoff section for $C^\infty$-generic contact forms on arbitrary closed contact 3-manifolds. Actually, our result will not be a perturbative one: it will apply to closed contact 3-manifolds satisfying the \emph{Kupka-Smale condition}, meaning that all closed Reeb orbits are non-degenerate (i.e.~their Floquet multipliers are not complex roots of unity), and the stable and unstable manifolds of the hyperbolic closed Reeb orbits intersect transversely.

\begin{MainThm}
\label{t:Birkhoff_section}
Any closed contact 3-manifold satisfying the Kupka-Smale condition admits a Birkhoff section for its Reeb flow.
\end{MainThm}

It is known that, for any $r\in\N\cup\{\infty\}$, a $C^r$ generic contact form satisfies the Kupka-Smale condition. This can be proved by extending the argument of Peixoto \cite{Peixoto:1967us} for general vector fields, and could also be obtained from the results of Robinson \cite{Robinson:1970tm} on generic Hamiltonian systems. Therefore, Theorem~\ref{t:Birkhoff_section} implies the following statement, that generalizes Theorem~\ref{t:CDHR}. We denote by $\FF^r(N)$ the space of smooth contact forms on a closed 3-manifold $N$, endowed with the $C^r$ topology.

\begin{Cor}
On any closed 3-manifold $N$, for any $r\in\N\cup\{\infty\}$, there exists a residual subset $\RR\subseteq\FF^r(N)$ such that, for every $\lambda\in\RR$, the Reeb flow of $(N,\lambda)$ admits a Birkhoff section.
\hfill\qed
\end{Cor}

According to a theorem of the first author together with Paternain \cite{Contreras:2002vb}, the Kupka-Smale condition also holds for geodesic vector fields of $C^r$ generic Riemannian metrics, which are a special class of Reeb vector fields. Therefore, Theorem~\ref{t:Birkhoff_section} also has the following corollary. We denote by $\GG^r(M)$ the space of smooth Riemannian metrics on a closed surface $M$, endowed with the $C^r$ topology.

\begin{Cor}
On any closed surface $M$, for any $r\in\N\cup\{\infty\}$, there exists a residual subset $\RR\subseteq\GG^r(M)$ such that, for every $g\in\RR$, the geodesic flow of $(M,g)$ admits a Birkhoff section.
\end{Cor}

Our proof of Theorem~\ref{t:Birkhoff_section} still employs the broken book decompositions of closed contact 3-manifolds, but not the other above mentioned ingredients employed by Colin, Dehornoy, Hryniewicz, and Rechtman. 
The pages of a broken book decomposition are surfaces of section for the Reeb flow and, as Colin, Dehornoy, and Rechtman showed in their work \cite{Colin:2020tl}, surgery techniques due to Fried \cite{Fried:1983uj} can be applied to produce a Birkhoff section out of the pages of the broken book, provided some broken binding component has transverse homoclinics in all its stable and unstable separatrices. In our proof, we show that indeed, under the Kupka-Smale condition, \emph{every} broken binding component has transverse homoclinics in all its stable and unstable separatrices. Furthermore, we prove the following remarkable fact (Theorem~\ref{t:homoclinics}): for any broken binding component $\gamma$, the closures of the stable manifold $\Ws(\gamma)$ coincides with the closure of the unstable manifold $\Wu(\gamma)$. We also provide a simpler construction of a Birkhoff section in the special case of Kupka-Smale geodesic flows of closed surfaces (Theorem~\ref{t:Birkhoff_section_geodesic_flow}).

\subsection{Organization of the paper}

In Section~\ref{s:preliminaries} we recall the main required notions from Reeb dynamics, and the definition of broken book decomposition of a closed contact 3-manifold. In Section~\ref{s:construction}, we discuss the results that allow to produce, under suitable conditions, a Birkhoff section out of the pages of a broken book decomposition. In Section~\ref{s:geodesic_flows}, we prove Theorem~\ref{t:Birkhoff_section} in the special case of Kupka-Smale geodesic flows of closed contact manifolds. Finally, in Section~\ref{s:Reeb}, we prove Theorem~\ref{t:Birkhoff_section} in full generality.

\subsection{Acknowledgements}

This work was completed while Marco Mazzucchelli was a visitor at the Fakult\"at f\"ur Mathematik of the Ruhr-Universit\"at Bochum, Germany. He would like to thank Alberto Abbondandolo, Stefan Suhr, and Kai Zehmisch for their hospitality.

\section{Preliminaries}
\label{s:preliminaries}

\subsection{Kupka-Smale contact 3-manifolds}
\label{ss:KS}
Let $(N,\lambda)$ be a closed contact 3-man\-i\-fold. The contact form $\lambda$ is a 1-form on $N$ that defines a volume form $\lambda\wedge d\lambda$. The associated Reeb vector field $X$ is defined by the equations $\lambda(X)\equiv1$ and $d\lambda(X,\cdot)\equiv0$. We denote by $\phi_t:N\to N$ its flow, which is called the Reeb flow.

Let $\gamma(t):=\phi_t(z_0)$ be a closed Reeb orbit, that is, $\gamma(t)=\gamma(t+t_0)$ for some minimal period $t_0>0$. The Floquet multipliers of $\gamma$ are the eigenvalues of the linearized map $d\phi_{t_0}(z_0)|_{\ker(\lambda)}$. Since $\phi_{t_0}$ preserves the contact form $\lambda$, and since $d\lambda$ is symplectic over the contact distribution $\ker(\lambda)$, the Floquet multipliers come in pairs $\sigma,\sigma^{-1}\in\C\setminus\{0\}$. For each positive integer $k>0$, the closed orbit $\gamma$ is \emph{non-degenerate} at period $kt_0$ when $\ker(d\phi_{kt_0}(z_0)-I)=X(z)$; namely, when $\sigma^k\neq1$. The contact manifold $(N,\lambda)$ is said to be non-degenerate (or, employing a Riemannian terminology, \emph{bumpy}), when its closed Reeb orbits are non-degenerate for all possible periods, that is, no Floquet multiplier is a root of unity.

The closed Reeb orbit $\gamma$ is 
\begin{itemize}
\item \emph{elliptic} when $\sigma,\sigma^{-1}\in S^1$,
\item \emph{positively hyperbolic} when $\sigma,\sigma^{-1}\in(0,1)\cup(1,\infty)$,
\item \emph{negatively hyperbolic} when $\sigma,\sigma^{-1}\in(-\infty,-1)\cup(-1,0)$.
\end{itemize}
When $\gamma$ is hyperbolic, the Floquet multiplier $\sigma$ with absolute value $|\sigma|<1$ is called the stable Floquet multiplier. The stable and unstable distributions along $\gamma$ are defined respectively as
\begin{align*}
 \Es(\gamma(t)) = \ker\big(d\phi_{t_0}(\gamma(t))-\sigma I\big),
 \qquad
 \Eu(\gamma(t)) = \ker\big(d\phi_{t_0}(\gamma(t))-\sigma^{-1} I\big).
\end{align*}
The stable and unstable manifolds of $\gamma$ are defined respectively as
\begin{align*}
 \Ws(\gamma)=\bigcup_{t\in\R/t_0\Z} \Ws(\gamma(t)),
 \qquad
 \Wu(\gamma)=\bigcup_{t\in\R/t_0\Z} \Wu(\gamma(t)),
\end{align*}
where 
\begin{align*}
 \Ws(\gamma(t)) & =\Big\{ z\in N\ \Big|\ \lim_{r\to\infty}d(\phi_r(z),\gamma(t+r))=0\Big\},
 \\
 \Wu(\gamma(t)) & =\Big\{ z\in N\ \Big|\ \lim_{r\to-\infty}d(\phi_r(z),\gamma(t+r))=0\Big\}.
\end{align*}
Here, $d:N\times N\to[0,\infty)$ denotes any Riemannian distance. The spaces $\Ws(\gamma(t))$ and $\Wu(\gamma(t))$ are smooth injectively immersed 1-dimensional submanifolds of $N$, transverse to the Reeb vector field $X$, and with tangent spaces at $\gamma(t)$ given by
\begin{align*}
 T_{\gamma(t)}\Ws(\gamma(t))=\Es(\gamma(t)),
 \qquad
 T_{\gamma(t)}\Wu(\gamma(t))=\Eu(\gamma(t)).
\end{align*}
Since $\phi_r(\Ws(\gamma(t)))=\Ws(\gamma(t+r))$ and $\phi_r(\Wu(\gamma(t)))=\Wu(\gamma(t+r))$, the stable and unstable manifolds $\Ws(\gamma)$ and $\Wu(\gamma)$ are smooth injectively immersed 2-dimensional submanifolds invariant under the Reeb vector field.

A closed contact 3-manifold $(N,\lambda)$ is said to satisfy the \emph{Kupka-Smale condition} when it is non-degenerate, and satisfies the transversality $\Ws(\gamma_1)\pitchfork\Wu(\gamma_2)$ for each pair of (not necessarily distinct) closed Reeb orbits $\gamma_1,\gamma_2$.

\subsection{Broken book decompositions}

A \emph{surface of section} for the Reeb flow of the closed contact 3-manifold $(N,\lambda)$ is an immersed compact surface with boundary  $\Sigma\looparrowright N$ whose interior $\interior(\Sigma)$ is embedded and transverse to the Reeb vector field $X$, and whose boundary $\partial\Sigma$ is tangent to $X$. A surface of section $\Sigma$ is called a \emph{Birkhoff section} when there exists $T>0$ such that, for each $z\in N$, we have $\phi_{t}(z)\in\Sigma$ for some $t\in[0,T]$.

Motivated by the quest of Birkhoff sections, and inspired by Hofer, Wysocki, and Zehnder's finite energy foliations~\cite{Hofer:1998vy},
Colin, Dehornoy, and Rechtman \cite{Colin:2020tl} introduced the notion of \emph{broken book decomposition} of a non-degenerate closed contact 3-manifold $(N,\lambda)$, which consists of the following data:
\begin{itemize}

\item A \emph{binding} $K=\Krad\cup\Kbr$, which is the disjoint union of the \emph{radial binding} $\Krad\subset N$ consisting of  a finite collection of closed Reeb orbits, and of the  \emph{broken binding} $\Kbr\subset N$ consisting of a finite collection of hyperbolic closed Reeb orbits.

\item A family $\F$ of compact surfaces of section, called the \emph{pages}, whose interiors foliate $N\setminus K$ and whose union of boundaries is precisely
\begin{align*}
 \bigcup_{\Sigma\in\F}\partial\Sigma=K.
\end{align*}

\item Finitely many \emph{rigid pages} $\Sigma_1,...,\Sigma_n\in\F$.

\end{itemize}
This data is required to satisfy the following properties:
\begin{itemize}
\item (\emph{Radial binding}) Close to a small segment of a radial binding component $\gamma\subset\Krad$, the pages arrive radially as in Figure~\ref{f:radial}(a). For any page $\Sigma$ whose boundary contains $\gamma$, there exists $T>0$ such that, for each $z\in\Sigma$ sufficiently close to $\gamma$, we have $\phi_t(z)\in\Sigma$ for some $t\in(0,T]$.

\item (\emph{Broken binding}) Close to a small segment of a broken binding component $\gamma'\subset\Kbr$, the pages arrive radially in four sectors, and hyperbolically in the four sectors in between, as in Figure~\ref{f:radial}(b). The pages in the four hyperbolic sectors are precisely those that intersect $\Ws(\gamma')\cup\Wu(\gamma')$

\begin{figure}
\begin{footnotesize}
\begingroup%
  \makeatletter%
  \providecommand\color[2][]{%
    \errmessage{(Inkscape) Color is used for the text in Inkscape, but the package 'color.sty' is not loaded}%
    \renewcommand\color[2][]{}%
  }%
  \providecommand\transparent[1]{%
    \errmessage{(Inkscape) Transparency is used (non-zero) for the text in Inkscape, but the package 'transparent.sty' is not loaded}%
    \renewcommand\transparent[1]{}%
  }%
  \providecommand\rotatebox[2]{#2}%
  \newcommand*\fsize{\dimexpr\f@size pt\relax}%
  \newcommand*\lineheight[1]{\fontsize{\fsize}{#1\fsize}\selectfont}%
  \ifx\svgwidth\undefined%
    \setlength{\unitlength}{345.13683734bp}%
    \ifx\svgscale\undefined%
      \relax%
    \else%
      \setlength{\unitlength}{\unitlength * \real{\svgscale}}%
    \fi%
  \else%
    \setlength{\unitlength}{\svgwidth}%
  \fi%
  \global\let\svgwidth\undefined%
  \global\let\svgscale\undefined%
  \makeatother%
  \begin{picture}(1,0.17746247)%
    \lineheight{1}%
    \setlength\tabcolsep{0pt}%
    \put(0.38856712,0.12332316){\color[rgb]{0,0,0}\makebox(0,0)[lt]{\lineheight{1.25}\smash{\begin{tabular}[t]{l}$\gamma$\end{tabular}}}}%
    \put(0,0){\includegraphics[width=\unitlength,page=1]{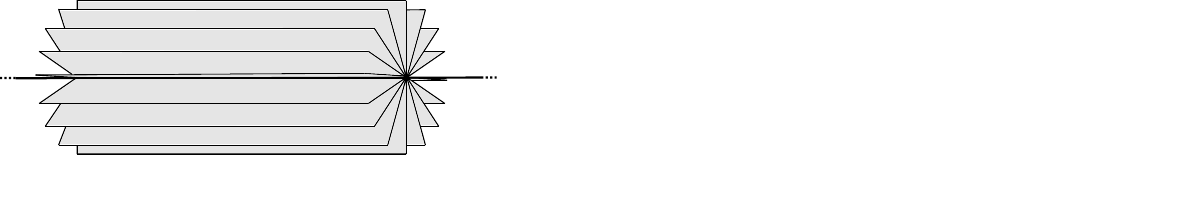}}%
    \put(0.95595048,0.12319808){\color[rgb]{0,0,0}\makebox(0,0)[lt]{\lineheight{1.25}\smash{\begin{tabular}[t]{l}$\gamma$\end{tabular}}}}%
    \put(0,0){\includegraphics[width=\unitlength,page=2]{z_radial.pdf}}%
    \put(0.17126205,0.00163205){\color[rgb]{0,0,0}\makebox(0,0)[lt]{\lineheight{1.25}\smash{\begin{tabular}[t]{l}$\textbf{(a)}$\end{tabular}}}}%
    \put(0.77971608,0.00163205){\color[rgb]{0,0,0}\makebox(0,0)[lt]{\lineheight{1.25}\smash{\begin{tabular}[t]{l}$\textbf{(b)}$\end{tabular}}}}%
  \end{picture}%
\endgroup%

\end{footnotesize}
\caption{\textbf{(a)} A radial binding component $\gamma$. \textbf{(b)} A broken binding component $\gamma'$.}
\label{f:radial}
\end{figure}

\item (\emph{Rigid pages}) Every Reeb orbit intersects at least once the collection of rigid pages, i.e.\ for each $z\in N$ there exists $t\in\R$ such that $\phi_t(z)\in\Sigma_1\cup...\cup\Sigma_n$. If $\phi_t(z)\not\in\Sigma_1\cup...\cup\Sigma_n$ for all $t>0$, then $z\in\Ws(\gamma)$ for some broken binding component $\gamma\subset\Kbr$. Analogously, if $\phi_t(z)\not\in\Sigma_1\cup...\cup\Sigma_n$ for all $t<0$, then $z\in\Wu(\gamma')$ for some broken binding component $\gamma'\subset\Kbr$. 

\end{itemize}
According to a theorem of Colin, Dehornoy, and Rechtman, any non-degenerate closed contact manifold admits a broken book decomposition. We refer the reader to \cite{Colin:2020tl} for a proof of this fact, as well as for more details concerning broken book decompositions.

When the broken binding $\Kbr$ is empty, the broken book decomposition reduces to an ordinary rational open book decomposition. In this case, any page $\Sigma$ is a Birkhoff section.

\subsection{Heteroclinics and homoclinics}
Let $\gamma_1,\gamma_2$ be two hyperbolic closed Reeb orbits. We recall that a \emph{heteroclinic} $\gamma$ from $\gamma_1$ to $\gamma_2$ is a Reeb orbit contained in $\Wu(\gamma_1)\cap\Ws(\gamma_2)\setminus(\gamma_1\cup\gamma_2)$. When the intersection $\Wu(\gamma_1)\cap\Ws(\gamma_2)$ is transverse at some point (and therefore at all points) of $\gamma$, we say that $\gamma$ is a \emph{transverse} heteroclinic. Notice that, under the Kupka-Smale assumption, every heteroclinic is automatically a transverse one. A \emph{homoclinic} is a heteroclinic from a closed Reeb orbit to itself.

The following lemma was originally proved by Hofer, Wysocki, and Zehnder \cite[Prop.~7.5]{Hofer:2003wf} in the context of finite energy foliations, and reproved in the general setting of broken book decompositions by Colin, Dehornoy, and Rechtman \cite[Lemma~4.1]{Colin:2020tl}.

\begin{Lemma}
\label{l:heteroclinics}
Let $(N,\lambda)$ be a closed contact 3-man\-i\-fold, equipped with a broken book decomposition with broken binding $\Kbr\neq\varnothing$. Let $\gamma\subset\Kbr$ be a broken binding component.
\begin{itemize}
\item[$(i)$] Every path-connected component $W'\subset\Wu(\gamma)\setminus\gamma$ contains a heteroclinic towards some broken binding orbit $\gamma'\subset\Kbr$, i.e.~$W'\cap\Ws(\gamma')\neq\varnothing$.

\item[$(ii)$] Every path-connected component $W''\subset\Ws(\gamma)\setminus\gamma$ contains a heteroclinic from some broken binding orbit $\gamma''\subset\Kbr$, i.e.~$W''\cap\Wu(\gamma'')\neq\varnothing$.\hfill\qed

\end{itemize}
\end{Lemma}

\section{Construction of a Birkhoff section}
\label{s:construction}

\subsection{Fried's surgery}
Let $(N,\lambda)$ be a closed contact 3-manifold with Reeb vector field $X$. In this paper, by \emph{immersed surface of section} for the Reeb vector field we mean an immersed compact surface
with boundary  $\Sigma\looparrowright N$ whose interior $\interior(\Sigma)$ is transverse to $X$ and whose boundary $\partial\Sigma$ is tangent to $X$.

The following lemma goes along the line of the arguments in Colin, Dehornoy, and Rechtman's \cite{Colin:2020tl} for the construction of a broken book decomposition, which in turn were based on a surgery technique due to Fried \cite{Fried:1983uj}.

\begin{Lemma}
\label{l:surgery}
Let $(N,\lambda)$ be a non-degenerate closed contact 3-manifold, equipped with a broken book decomposition with binding $K=\Krad\cup\Kbr$. Assume that there exists a broken binding component $\gamma\subset\Kbr$, and an immersed surface of section $\Sigma\looparrowright N$ whose interior $\interior(\Sigma)$ intersects $\gamma$, and whose boundary $\partial\Sigma$ is disjoint from the binding $K$. Then, there exists a broken book decomposition with broken binding 
 $\Kbr\setminus\gamma$.
\end{Lemma}

\begin{proof}
We denote by $\Sigma_1,...,\Sigma_n\subset N$ the rigid pages of the broken book.
We perturb the interior of the surface of section $\Sigma$ while keeping its boundary fixed, in such a way to obtain a new surface of section $\Sigma'$ that is $C^1$-close to $\Sigma$, has the same the boundary $\partial\Sigma'=\partial\Sigma$, and has self-intersections and intersections with the rigid pages $\Sigma_1,...,\Sigma_n$ in general position. Since $\Sigma'$ and $\Sigma$ are $C^1$-close, $\interior(\Sigma')$ still intersects transversely the broken binding component $\gamma$. We denote by $P\subset\Sigma'$ the subset consisting of the points of self-intersections of $\Sigma'$, and apply a surgery technique due to Fried  \cite[Section~2]{Fried:1983uj} in order to resolve such self-intersections: we resolve the lines of double points in $P$ as in Figure~\ref{f:surgery}(a), the isolated triple intersections in $P$ as in Figure~\ref{f:surgery}(b), and the lines of double points in $P$ with one strand ending at the boundary of $\Sigma'$ as in Figure~\ref{f:surgery}(c); it remains to consider the case 
of lines of double points in $P$ in which both strands end in the same boundary component $\zeta\in\Sigma'$: as it was pointed out in \cite[proof or Corollary~3.2]{Colin:2020tl}, once we resolved the 
double points outside a small tubular neighborhood $W$ of $\zeta$, depending on the trace of the obtained surface of section on $\partial W$, we extend it within $W$ by attaching a suitable finite union of annuli with boundary on $\zeta$ or a suitable finite union of meridional disks (and, in this case, $\zeta$ will not be a boundary component of the obtained surface of section anymore).
We denote by $\Sigma''$ the obtained surface of section, whose interior $\interior(\Sigma'')$ is embedded in $N$ and still intersects $\gamma$. We perturb $\interior(\Sigma'')$ so that the obtained surface of section $\Sigma'''$ intersects $\Sigma_1\cup...\cup\Sigma_n$ in general position, and $\interior(\Sigma''')$ still intersects $\gamma$. Next, we resolve the intersections   $Q:=\Sigma'''\cap(\Sigma_1\cup...\cup\Sigma_n)$ as in Figure~\ref{f:surgery}(a) and 
 Figure~\ref{f:surgery}(c), and with the above mentioned procedure to deal with the lines of double points in which both strands end in the same boundary component. 
 This procedure replaces $\Sigma''',\Sigma_1,...,\Sigma_n$ 
 with another finite collection of surfaces of section $\Upsilon_1,...,\Upsilon_{m}$ 
whose interiors $\interior(\Upsilon_i)$ are 
embedded in $N$, and $\interior(\Upsilon_i)\cap\interior(\Upsilon_j)=\varnothing$ for $i\neq j$. The surgery only modifies the surfaces of section $\Sigma''',\Sigma_1,...,\Sigma_n$ within an arbitrarily small neighborhood $U\subset N$ of $Q$, i.e.
\begin{align*}
(\Upsilon_1\cup...\cup\Upsilon_{m})\setminus U=(\Sigma'''\cup\Sigma_1\cup...\cup\Sigma_n)\setminus U.
\end{align*}
Therefore, there exist constants $t_2>t_1>0$ such that, for each $z\in N$, if the orbit segment $\phi_{[-t_1,t_1]}(z)$ intersects $\Sigma'''\cup\Sigma_1\cup...\cup\Sigma_n$, then the larger orbit segment $\phi_{[-t_2,t_2]}(z)$ intersects $\Upsilon_1\cup...\cup\Upsilon_m$. We refer to this property as to the \emph{intersection property}.

\begin{figure}
\begin{footnotesize}
\begingroup%
  \makeatletter%
  \providecommand\color[2][]{%
    \errmessage{(Inkscape) Color is used for the text in Inkscape, but the package 'color.sty' is not loaded}%
    \renewcommand\color[2][]{}%
  }%
  \providecommand\transparent[1]{%
    \errmessage{(Inkscape) Transparency is used (non-zero) for the text in Inkscape, but the package 'transparent.sty' is not loaded}%
    \renewcommand\transparent[1]{}%
  }%
  \providecommand\rotatebox[2]{#2}%
  \newcommand*\fsize{\dimexpr\f@size pt\relax}%
  \newcommand*\lineheight[1]{\fontsize{\fsize}{#1\fsize}\selectfont}%
  \ifx\svgwidth\undefined%
    \setlength{\unitlength}{341.66565774bp}%
    \ifx\svgscale\undefined%
      \relax%
    \else%
      \setlength{\unitlength}{\unitlength * \real{\svgscale}}%
    \fi%
  \else%
    \setlength{\unitlength}{\svgwidth}%
  \fi%
  \global\let\svgwidth\undefined%
  \global\let\svgscale\undefined%
  \makeatother%
  \begin{picture}(1,1.02751658)%
    \lineheight{1}%
    \setlength\tabcolsep{0pt}%
    \put(0,0){\includegraphics[width=\unitlength,page=1]{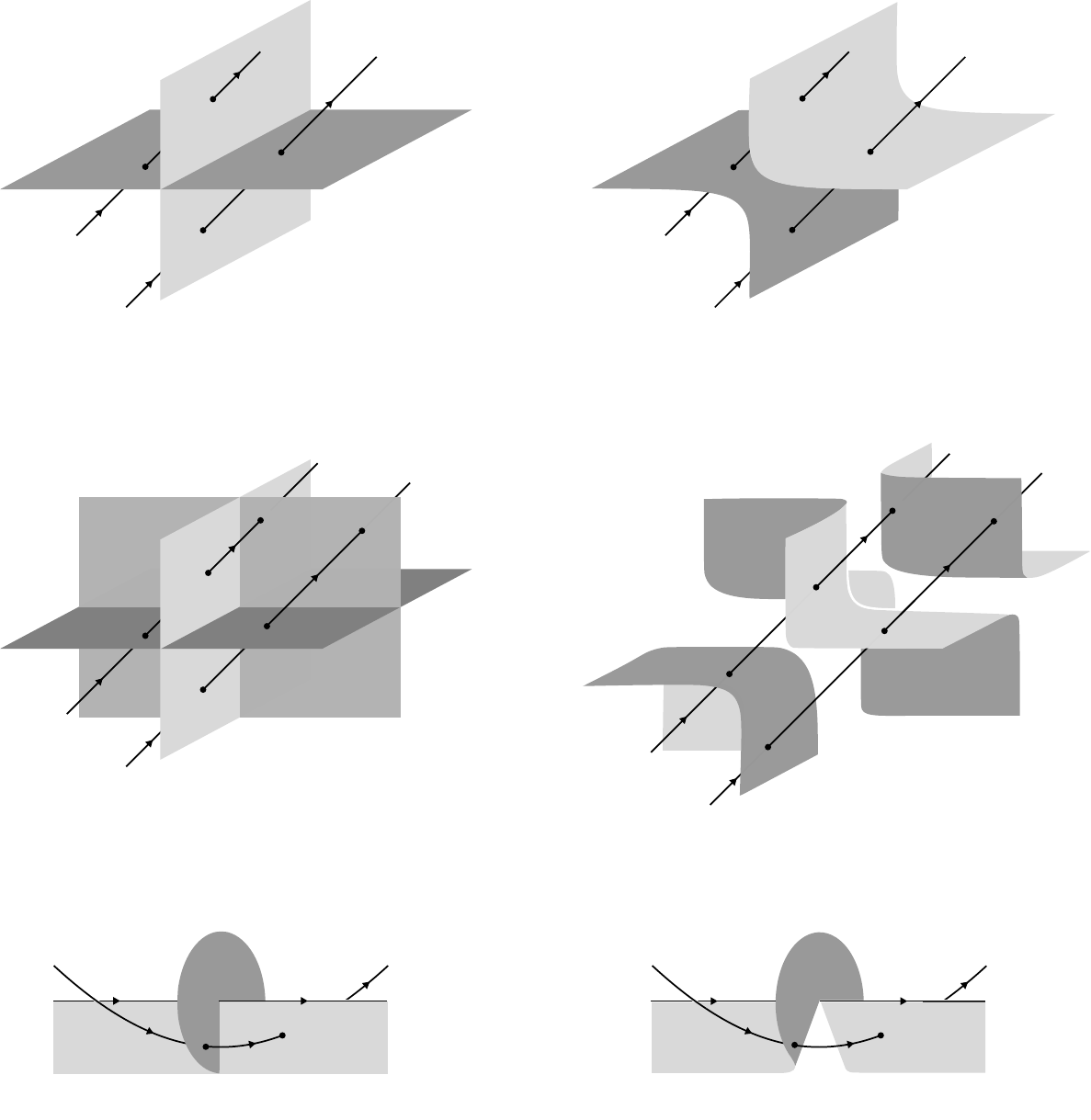}}%
    \put(0.46310342,0.86520082){\color[rgb]{0,0,0}\makebox(0,0)[lt]{\lineheight{1.25}\smash{\begin{tabular}[t]{l}{\large$\Rightarrow$}\end{tabular}}}}%
    \put(0.46310342,0.4700778){\color[rgb]{0,0,0}\makebox(0,0)[lt]{\lineheight{1.25}\smash{\begin{tabular}[t]{l}{\large$\Rightarrow$}\end{tabular}}}}%
    \put(0.46310342,0.07495455){\color[rgb]{0,0,0}\makebox(0,0)[lt]{\lineheight{1.25}\smash{\begin{tabular}[t]{l}{\large$\Rightarrow$}\end{tabular}}}}%
    \put(0.46310342,0.7203223){\color[rgb]{0,0,0}\makebox(0,0)[lt]{\lineheight{1.25}\smash{\begin{tabular}[t]{l}\textbf{(a)}\end{tabular}}}}%
    \put(0.46310342,0.25934565){\color[rgb]{0,0,0}\makebox(0,0)[lt]{\lineheight{1.25}\smash{\begin{tabular}[t]{l}\textbf{(b)}\end{tabular}}}}%
    \put(0.46310342,0.00471038){\color[rgb]{0,0,0}\makebox(0,0)[lt]{\lineheight{1.25}\smash{\begin{tabular}[t]{l}\textbf{(c)}\end{tabular}}}}%
  \end{picture}%
\endgroup%

\end{footnotesize}
\caption{Fried's surgeries, as described in \cite{Fried:1983uj}. 
Observe that the choice of the resolutions, being transversal to the flow, are unique. }
\label{f:surgery}
\end{figure}
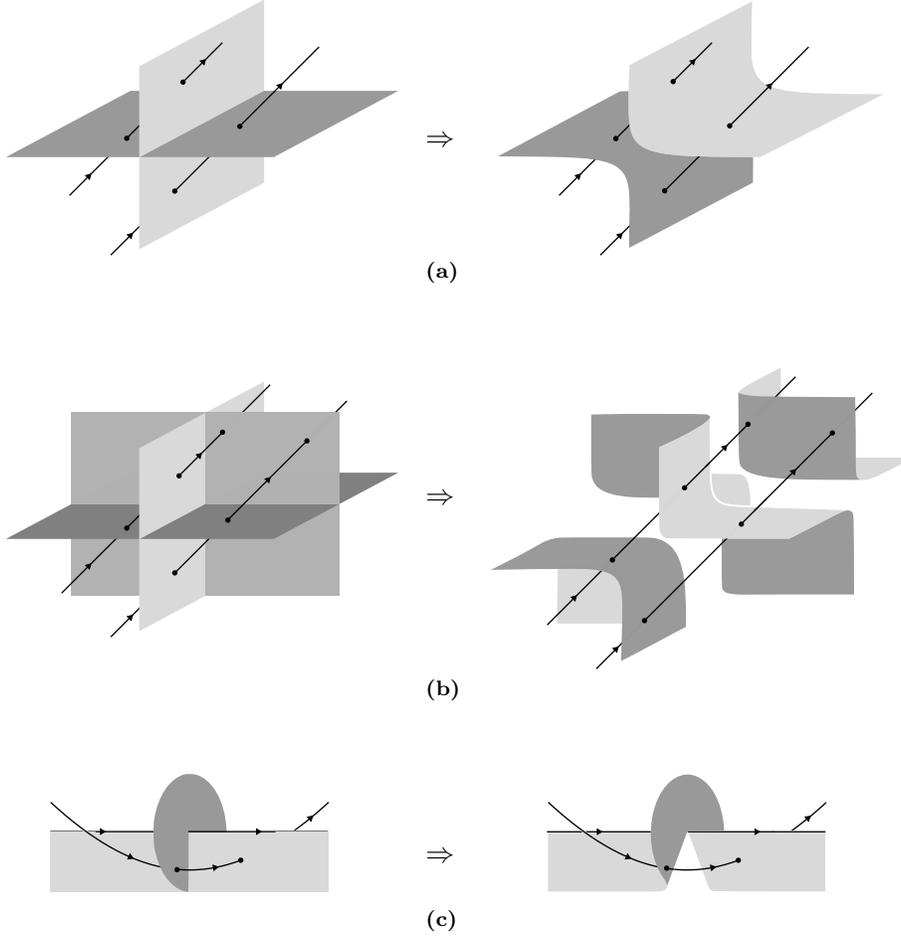

We say that a component $\zeta\subset\partial \Upsilon_1\cup...\cup\partial \Upsilon_m$ 
has the {\it radial property} if there exists $T>0$
such that, for every $z\in N$ sufficiently close to $\zeta$, the segment
$\phi_{(0,T]}(z)$ intersects $\Upsilon_1\cup...\cup\Upsilon_m$.
Since every Reeb orbit $t\mapsto\phi_t(z)$ intersects $\Sigma_1\cup...\cup\Sigma_n$, 
the intersection property implies that it intersects $\Upsilon_1\cup...\cup\Upsilon_m$ as well. 
Since $\interior(\Sigma''')$ intersects $\gamma$ transversely, the forward orbit of  any $z\in N$
sufficiently close to $\gamma$ intersects $\Sigma'''$ in a time close to the period of $\gamma$. This, together with the intersection property, implies the radial property for $\gamma$.

Consider a component 
$\zeta\subset \partial\Sigma'''\cap(\partial \Upsilon_1\cup...\cup\partial \Upsilon_m)$.
Since $\zeta\cap K=\varnothing$, we have that $\zeta$ intersects transversely 
$\Sigma_1\cup...\cup \Sigma_n$. The forward orbit of any $z\in N$ sufficiently 
close to $\zeta$ intersects $\Sigma_1\cup...\cup \Sigma_n$ in a time close to the period
of $\zeta$. This, together with the intersection property,  implies the radial property for $\zeta$.
Finally, the intersection property alone implies that each radial binding component $\zeta'\subset\Krad\cap (\partial \Upsilon_1\cup...\cup\partial \Upsilon_m)$ satisfies the 
radial property.

Now, out of the finite family of surfaces of section $\Upsilon_1,...,\Upsilon_n$, the argument in Colin, Dehornoy and Rechtman's \cite[Proof of Theorem~1.1]{Colin:2020tl} provide a broken book decomposition of $(N,\lambda)$ with binding $\partial \Upsilon_1\cup...\cup\partial \Upsilon_m\subseteq K\cup\partial\Sigma'''$. The conclusions of the last two paragraphs imply that the radial binding of the new broken book decomposition is contained in $\Krad\cup\partial\Sigma'''\cup\gamma$, and  the broken binding of the new broken book decomposition is $\Kbr\setminus\gamma$.
\end{proof}

\subsection{Transverse homoclinics in all separatrices}
\label{ss:transverse_homoclinics_in_all_separatrices}
Let $(N,\lambda)$ be a closed contact 3-manifold with Reeb vector field $X$ and Reeb flow $\phi_t:N\to N$. Assume that there exists a hyperbolic closed Reeb orbit $\gamma$. We fix a point $z_0\in\gamma$, and an embedded open disk $D\subset N$ transverse to $X$ and containing the point $z_0$. The intersections $\Wu(\gamma)\cap D$ and $\Ws(\gamma)\cap D$ are transverse, and we denote by $\lu\subset\Wu(\gamma)\cap D$ and $\ls\subset\Ws(\gamma)\cap D$ the 
path-connected components containing $z_0$. Up to shrinking $D$ around $z_0$, both $\lu$ and $\ls$ are embedded 1-dimensional manifolds 
 intersecting only at $z_0$,
 and both separating $D$ into two path-connected components.

We write the complements $\lu\setminus\{z_0\}$ and $\ls\setminus\{z_0\}$ 
as union of path-connected components
\begin{align*}
 \lu\setminus\{z_0\}=\lu_1\cup\lu_2,
 \qquad
 \ls\setminus\{z_0\}=\ls_1\cup\ls_2.
\end{align*}
The open intervals $\lu_i$ and $\ls_i$ are called unstable separatrices and stable separatrices respectively. We say that $\gamma$ has \emph{transverse homoclinics in both unstable separatrices} when there are points of transverse  intersection $z\in\lu_1\cap\Ws(\gamma)$ and $z'\in\lu_2\cap\Ws(\gamma)$. 
Analogously, we say that $\gamma$ has \emph{transverse homoclinics in both stable separatrices} when there are points of transverse  intersection  $z\in\ls_1\cap\Wu(\gamma)$ and $z'\in\ls_2\cap\Wu(\gamma)$.

In Section~\ref{s:geodesic_flows}, we shall need the following consequence of Lemma~\ref{l:heteroclinics}, originally proved by Colin, Dehornoy, and Rechtman \cite[Lemma~4.2]{Colin:2020tl}. The statement requires the considered closed contact 3-manifold to satisfy the Kupka-Smale condition (see Section~\ref{ss:KS}).

\begin{Lemma}[Colin-Dehornoy-Rechtman] 
\label{l:two_homoclinics}
Let $(N,\lambda)$ be a closed contact 3-man\-i\-fold satisfying the Kupka-Smale condition, equipped with a broken book decomposition with non-empty broken binding $\Kbr\neq\varnothing$. Then, there exists a broken binding component $\gamma\subset\Kbr$ with transverse homoclinics in both stable separatrices. Analogously, there exists a broken binding component $\gamma'\subset\Kbr$ with transverse homoclinics in both unstable separatrices.
\hfill\qed
\end{Lemma}

\begin{figure}
\begin{footnotesize}
\begingroup%
  \makeatletter%
  \providecommand\color[2][]{%
    \errmessage{(Inkscape) Color is used for the text in Inkscape, but the package 'color.sty' is not loaded}%
    \renewcommand\color[2][]{}%
  }%
  \providecommand\transparent[1]{%
    \errmessage{(Inkscape) Transparency is used (non-zero) for the text in Inkscape, but the package 'transparent.sty' is not loaded}%
    \renewcommand\transparent[1]{}%
  }%
  \providecommand\rotatebox[2]{#2}%
  \newcommand*\fsize{\dimexpr\f@size pt\relax}%
  \newcommand*\lineheight[1]{\fontsize{\fsize}{#1\fsize}\selectfont}%
  \ifx\svgwidth\undefined%
    \setlength{\unitlength}{364.30660829bp}%
    \ifx\svgscale\undefined%
      \relax%
    \else%
      \setlength{\unitlength}{\unitlength * \real{\svgscale}}%
    \fi%
  \else%
    \setlength{\unitlength}{\svgwidth}%
  \fi%
  \global\let\svgwidth\undefined%
  \global\let\svgscale\undefined%
  \makeatother%
  \begin{picture}(1,0.42493772)%
    \lineheight{1}%
    \setlength\tabcolsep{0pt}%
    \put(0,0){\includegraphics[width=\unitlength,page=1]{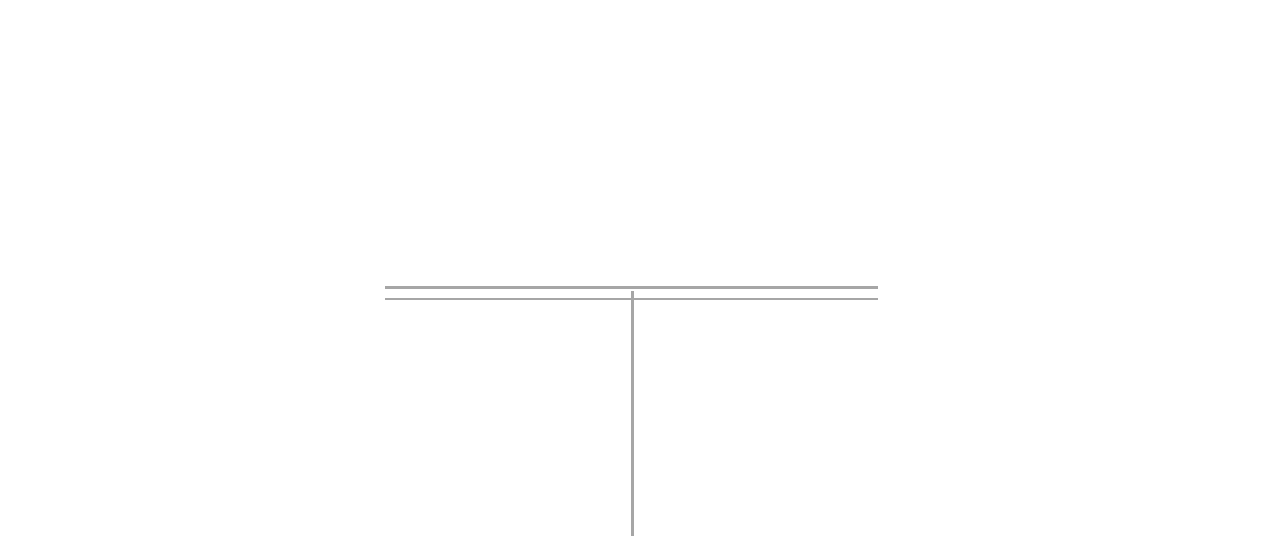}}%
    \put(0.56348354,0.40296961){\color[rgb]{0,0,0}\makebox(0,0)[lt]{\lineheight{1.25}\smash{\begin{tabular}[t]{l}$\gray{\ell_{n_2}'}$\end{tabular}}}}%
    \put(0,0){\includegraphics[width=\unitlength,page=2]{z_homoclinic.pdf}}%
    \put(0.6977991,0.18886415){\color[rgb]{0,0,0}\makebox(0,0)[lt]{\lineheight{1.25}\smash{\begin{tabular}[t]{l}$\lu_i$\end{tabular}}}}%
    \put(0.49192841,0.40296961){\color[rgb]{0,0,0}\makebox(0,0)[lt]{\lineheight{1.25}\smash{\begin{tabular}[t]{l}$\ls_j$\end{tabular}}}}%
    \put(0.50839819,0.20533358){\color[rgb]{0,0,0}\makebox(0,0)[lt]{\lineheight{1.25}\smash{\begin{tabular}[t]{l}$z_0$\end{tabular}}}}%
    \put(0.6977991,0.25886039){\color[rgb]{0,0,0}\makebox(0,0)[lt]{\lineheight{1.25}\smash{\begin{tabular}[t]{l}$\gray{\ell_{n_1}}$\end{tabular}}}}%
    \put(-0.00216057,0.25886039){\color[rgb]{0,0,0}\makebox(0,0)[lt]{\lineheight{1.25}\smash{\begin{tabular}[t]{l}$ $\end{tabular}}}}%
    \put(0,0){\includegraphics[width=\unitlength,page=3]{z_homoclinic.pdf}}%
  \end{picture}%
\endgroup%

\end{footnotesize}
\caption{Transverse homoclinics in all the separatrices}
\label{f:homoclinics}
\end{figure}

Assume now that $\gamma$ has \emph{transverse homoclinics in all the separatrices}, meaning that for each $i,j\in\{1,2\}$ there exist points of transverse intersections $z\in\lu_i\cap\Ws(\gamma)$ and $z'\in\ls_j\cap\Wu(\gamma)$.
The existence of such $z$ and $z'$  implies the existence of sequences 
\[
z_n=\phi_{-t_n}(z)\in\lu_i\cap\Ws(\gamma),
\qquad
z_n'=\phi_{t_n'}(z)\in\ls_j\cap\Wu(\gamma)\] 
such that $t_n\to\infty$ and $t_n'\to\infty$. Notice that $z_n\to z_0$ and $z_n'\to z_0$. We denote by $\ell_n'\subset\Ws(\gamma)\cap D$ the path-connected component containing $z_n$, and by $\ell_n\subset\Wu(\gamma)\cap D$ the path-connected component containing $z_n'$. By the $\lambda$-lemma from hyperbolic dynamics \cite[Prop.~6.2.23]{Katok:1995th}, the sequence $\ell_n$ accumulates on $\lu$, and the sequence $\ell_n'$ accumulates on $\ls$. Therefore, for $n_1$ and $n_2$ large enough, we have a non-empty transverse intersection \[\ell_{n_1}\cap\ell_{n_2}'\neq\varnothing,\]
see Figure~\ref{f:homoclinics}.

Let $b_{\min}=b_{\min}(N,\lambda)\geq0$ be the minimal number of broken binding components of a broken book decomposition of $(N,\lambda)$; the existence of a Birkhoff section is equivalent to $b_{\min}=0$. We say that a broken book decomposition is \emph{minimal} when it has precisely $b_{\min}$ broken binding components. The following statement is implicit in Colin, Dehornoy, and Rechtman's \cite[Section~4]{Colin:2020tl}, and is based on a construction due to Fried \cite{Fried:1983uj}.

\begin{Lemma}
\label{l:no_homoclinics_in_all_separatrices}
Let $(N,\lambda)$ be a non-degenerate closed contact 3-manifold, equipped with a minimal broken book decomposition. Then no component of the broken binding has transverse homoclinics in all the separatrices.
\end{Lemma}

\begin{proof}
Let $K=\Krad\cup\Kbr$ be the binding of the broken book. We assume by contradiction that a broken binding component $\gamma\subset\Kbr$ has transverse homoclinics in all the separatrices. We shall employ a construction due to Fried \cite[Sect.~2]{Fried:1983uj}.

\begin{figure}
\begin{footnotesize}
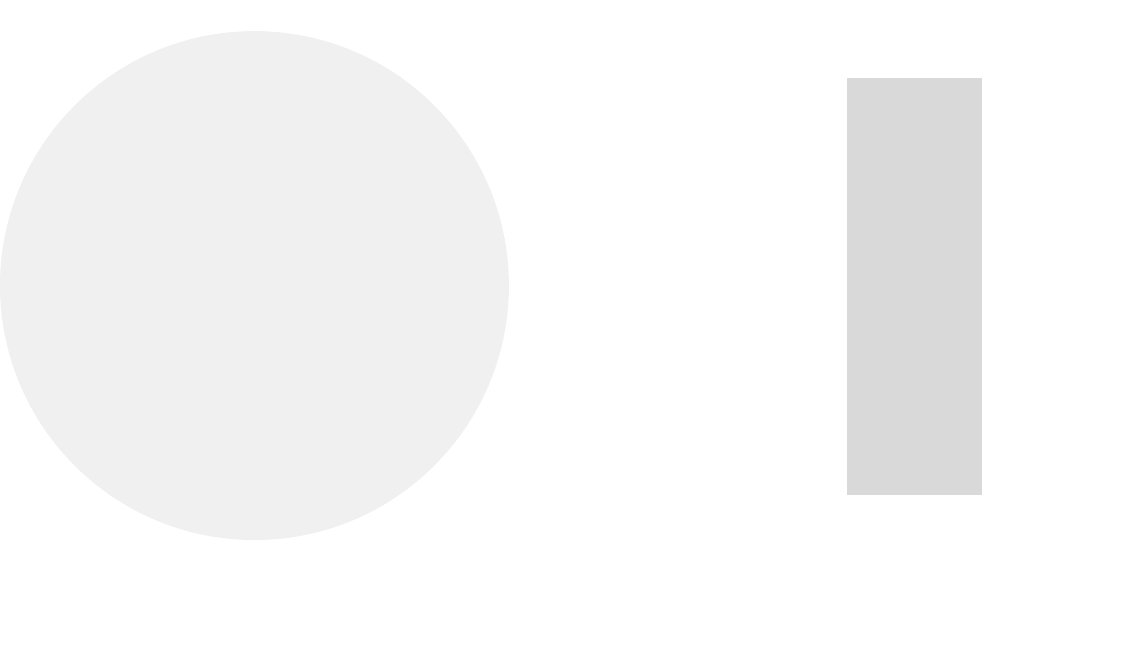
\end{footnotesize}
\caption{}
\label{f:transverse}
\end{figure}

We fix any point $z_0\in\gamma$ and a small embedded open disk $D\subset N$ transverse to the Reeb vector field $X$ and containing $z_0$. In particular, we require $D$ to be small enough so that 
\begin{align}
\label{e:small_disk}
D\cap K=\{z_0\}. 
\end{align}
We denote by $\ls\subset\Ws(\gamma)\cap D$ and $\lu\subset\Wu(\gamma)\cap D$ the path-connected components containing $z_0$, and we write $\ls\setminus\{z_0\}$ and $\lu\setminus\{z_0\}$ as the disjoint union of separatrices \[
\lu\setminus\{z_0\}=\lu_1\cup\lu_2,
\qquad
\ls\setminus\{z_0\}=\ls_1\cup\ls_2.
\] 
The fact that $\gamma$ has transverse homoclinics in all the separatrices implies that, up to switching the names of $\ls_1$ and $\ls_2$, there exist points $z_i\in\lu_i$ arbitrarily close to $z_0$, and arbitrarily large positive numbers $t_i>0$ such that $\phi_{t_i}(z_i)\in\ls_i$. By the implicit function theorem, there exists a maximal connected open subsets $U_i\subset D$ containing $z_i$ and smooth functions $\tau_i:U_i\to(0,\infty)$ such that $\tau_i(z_i)=t_i$ and $\psi_i(z):=\phi_{\tau_i(z)}(z)\in D$ for all $z\in U_i$.

We denote by $q_i\subset \Ws(\gamma)\cap D$ the path connected component containing $z_i$, and by $h_i\subset \Wu(\gamma)\cap D$ the path connected component containing $\phi_{t_i}(z_i)$. Up to choosing $t_i$ large enough, there exist points of transverse intersection
\begin{align*}
 w_{i,j}\in q_i\cap h_j,\qquad \forall i,j\in\{1,2\}.
\end{align*}
see Figure~\ref{f:transverse}(a). We denote by $q_i'\subset q_i$  the compact interval with boundary $\partial q_i'=\{w_{i,1},w_{i,2}\}$, and we set $q_i'':=\psi_i(q_i')$. Notice that we can make the heteroclinic points $w_{i,1}$ and $w_{i,2}$ arbitrarily close to $z_i$, and therefore the segment $q_i'$ to be arbitrarily short, by choosing $t_i$ to be large. In particular, we require $t_i$ to be large enough so that
\begin{align}
\nonumber 
q_i'\subset U_i,&\qquad \forall i\in\{1,2\},
\\
\label{e:wij}
w_{i,j}\in\psi_j(U_j),&\qquad\forall i,j\in\{1,2\},\\
\label{e:disjoint_flowout}
\phi_{t}(q_i'') \cap q_i''=\varnothing,& \qquad\forall t>0.
\end{align}

Since $\gamma$ has homoclinics in all the separatrices, by the $\lambda$-lemma the stable manifold $\Ws(\gamma)$ accumulates on $q_i$ from both sides, and the unstable manifold $\Wu(\gamma)$ accumulates on $\ls_i$ from both sides. Therefore, $z_i$ is contained in a heteroclinic rectangle $R_i$ that is an arbitrarily small neighborhood of $q_i'$. Namely, $R_i\subset D$ is an open disk containing $q_i'$ and with boundary 
\[\partial R_i=q_{i,1}\cup q_{i,2}\cup h_{i,1}\cup h_{i,2},\] where  $q_{i,1}, q_{i,2}\subset \Ws(\gamma)$ and $h_{i,1}, h_{i,2}\subset \Wu(\gamma)$ as in Figure~\ref{f:transverse}(b). We set
\begin{align*}
 R_i':=\psi_i(R_i).
\end{align*}
By~\eqref{e:wij}, we can choose the vertical edges $q_{i,1}, q_{i,2}\subset R_i$ so that
\begin{align*}
w_{1,j},w_{2,j}\in R_j',\qquad\forall j\in\{1,2\}.
\end{align*}
Finally, by~\eqref{e:disjoint_flowout}, we can choose the horizontal edges $h_{i,1}, h_{i,2}\subset R_i$ to be sufficiently close to $\partial q_i'$ so that
\begin{align*}
\phi_{t}(R_i')\cap R_i' = \varnothing,\qquad
\forall t>0.
\end{align*}
The heteroclinic rectangles $R_1,R_2$ and their images $R_1',R_2'$ intersect as in Figure~\ref{f:rectangles}.

\begin{figure}
\begin{footnotesize}
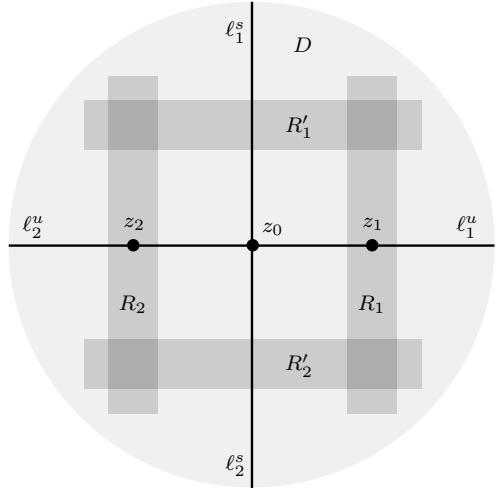
\end{footnotesize}
\caption{The heteroclinic rectangles $R_1, R_2, R_1', R_2'$.}
\label{f:rectangles}
\end{figure}
  
We consider the smooth diffeomorphism
\begin{align*}
 \psi:R_1\cup R_2\to R_1'\cup R_2', \qquad \psi|_{W_1}=\psi_1,\ \ \psi|_{W_2}=\psi_2,
\end{align*} 
Notice that $\psi$ is the first arrival map from $R_1\cup R_2$ to $R_1'\cup R_2'$. Namely, we have
\begin{align*}
 \psi(z)=\phi_{\tau(z)}(z),\qquad\forall z\in R_1\cup R_2,
\end{align*}
where
\begin{align*}
 \tau(z)=\min\big\{ t>0\ \big|\ \phi_t(z)\in R_1'\cup R_2'\big\}, \qquad\forall z\in R_1\cup R_2.
\end{align*}

\begin{figure}
\begin{footnotesize}
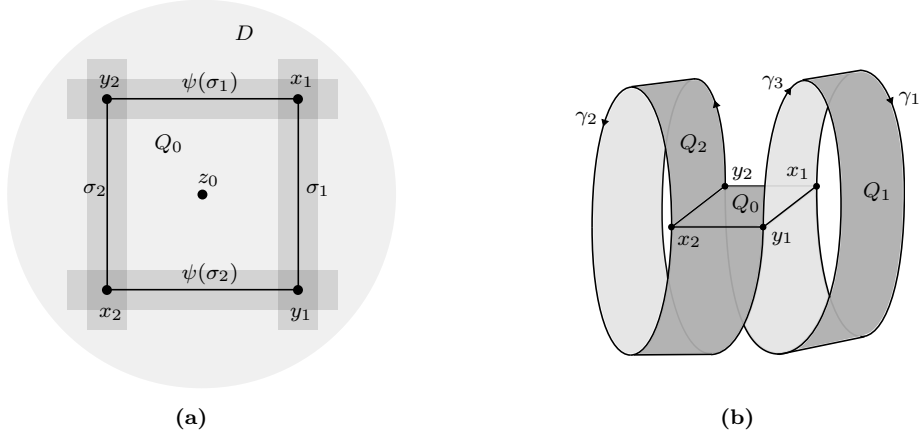
\end{footnotesize}
\caption{Fried's pair of pants.}
\label{f:fried}
\end{figure}

Since the intersections $R_i\cap R_j'$ contain the transverse homoclinics $w_{i,j}$,  we can employ symbolic dynamics as follows. We consider the compact invariant subset
\begin{align*}
 \Lambda = \bigcap_{n\in\Z} (\psi^n)^{-1}(R_1\cup R_2)\subset \bigcup_{i,j=1,2}\big( R_i\cap R_j' \big).
\end{align*}
This invariant subset is a horseshoe: there is a homeomorphism 
\begin{align*}
\kappa:\Lambda\toup^{\cong}\{1,2\}^{\Z},
\qquad
 \kappa(z)_n=
 \left\{
  \begin{array}{@{}ll}
    1, & \mbox{if }\psi^n(z)\in R_1,\vspace{5pt}\\ 
    2, & \mbox{if }\psi^n(z)\in R_2, \\ 
  \end{array} 
 \right.
\end{align*}
which conjugates the dynamics on $\Lambda$ according to the following commutative diagram:
\[\begin{tikzcd}[row sep=large]
 \Lambda
 \arrow[r, "\kappa"] 
 \arrow[d, "\psi"']
 &
 \{1,2\}^{\Z} 
 \arrow[d, "\shift"]
 \\
 \Lambda
 \arrow[r, "\kappa"] 
 &
 \{1,2\}^{\Z} 
\end{tikzcd}\]
Here, $\shift(a_n)=(a_{n+1})$. We consider the periodic words 
\begin{align*}
 \bm1 & =(...,1,1,1,1,...),\\
 \bm2 & =(...,2,2,2,2,...),\\
 \bm a & =(...,a_{-1},a_0,a_1,a_2,...), \mbox{ with } a_{2n}=1,\ a_{2n+1}=2,\ \forall n\in\Z.
\end{align*}
The corresponding points 
\begin{align*}
x_1 & :=\kappa^{-1}(\bm1)\in R_1\cap R_1',\\
x_2 & :=\kappa^{-1}(\bm2)\in R_2\cap R_2', \\
y_1 & :=\kappa^{-1}(\bm a)\in R_2'\cap R_1
\end{align*}
lie on closed Reeb orbits. If $\gamma$ is negatively hyperbolic, with an unlucky choice of the point $z_2$ and of the corresponding rectangle $R_2$, the closed Reeb orbits through $x_1$ and $x_2$ may coincide, but we can easily avoid this by  replacing the point $z_2$ with a transverse homoclinic point in $\lu_2\cap\Ws(\gamma)$ closer to $z_0$, so that the rectangle $R_2$ does not intersect the closed Reeb orbit through $x_1$. Therefore the closed Reeb orbits $\gamma_1(t)=\phi_t(x_1)$ and $\gamma_2(t)=\phi_t(x_2)$ are distinct, and intersect $D$ only in $x_1$ and $x_2$ respectively. On the other hand, the closed Reeb orbit $\gamma_3(t):=\phi_t(y_1)$ intersects $D$ in exactly two points: $y_1$ and
$y_2:=\kappa^{-1}(\shift(\bm a))\in R_1'\cap R_2$.
We consider a compact disk $Q_0\subset D$ with piecewise smooth boundary 
\[\partial Q_0=\sigma_1\cup \psi(\sigma_1)\cup \sigma_2\cup\psi(\sigma_2),\]
where each $\sigma_i\subset R_i$ is a smooth path joining $x_i$ and $y_i$, see Figure~\ref{f:fried}(a). Next, we consider the strips
\begin{align*}
Q_i:=\big\{ \phi_t(z)\ \big|\ z\in\sigma_i,\ t\in[0,\tau(z)] \big\}.
\end{align*}
The union $\Upsilon:=Q_0\cup Q_1\cup Q_2$ is a piecewise smooth pair of pants immersed in $N$. As a topological manifold, it has boundary $\partial\Upsilon=\gamma_1\cup\gamma_2\cup\gamma_3$ that is embedded in $N$, see Figure~\ref{f:fried}(b). Notice that $\interior(Q_0)$ is transverse to the Reeb vector field $X$, whereas $Q_1$ and $Q_2$ are tangent to $X$.

With a perturbation of $\interior(\Upsilon)$, we obtain a smooth immersed pair of pants $\Sigma$ with the same boundary $\partial\Sigma=\partial\Upsilon$, and with interior $\interior(\Sigma)$ that is transverse to $X$. Since $\partial\Sigma$ is embedded in $N$, it has a neighborhood $U\subset\Sigma$ that is embedded into $N$ as well. By~\eqref{e:small_disk}, $\partial\Sigma$ is disjoint from the binding $K$ of the broken book. Therefore, we can apply Lemma~\ref{l:surgery}, which provides a new broken book decomposition of $(N,\lambda)$ with broken binding $\Kbr\setminus\gamma$. This contradicts the fact that the original broken book decomposition was minimal.
\end{proof}

\begin{Lemma}
\label{l:positively_hyperbolic}
Let $(N,\lambda)$ be a non-degenerate closed contact 3-manifold, equipped with a minimal broken book decomposition. If $\gamma\subset\Kbr$ is a broken binding component admitting transverse homoclinics, then $\gamma$ is positively hyperbolic.
\end{Lemma}

\begin{proof}
Let us assume by contradiction that a broken binding component $\gamma\subset\Kbr$ is negatively hyperbolic and has transverse homoclinics. We denote by $\sigma\in(-1,0)$ and $\sigma^{-1}\in(-\infty,-1)$ the Floquet multipliers of $\gamma$.
We fix any point $z_0\in\gamma$, and consider the minimal period $t_0$ of $\gamma$. The tangent space $T_{z_0}N$ splits as
\begin{align*}
T_{z_0}N=\Es(z_0)\oplus\Eu(z_0)\oplus X(z_0),
\end{align*}
where $\Es(z_0)=\ker(d\phi_{t_0}(z_0)-\sigma I)$ and $\Eu(z_0)=\ker(d\phi_{t_0}(z_0)-\sigma^{-1} I)$. We consider a small embedded open disk $D\subset N$ containing $z_0$, with tangent space
\begin{align}
\label{e:tangent_space}
 T_{z_0}D=\Es(z_0)\oplus\Eu(z_0).
\end{align}
As usual, we require $D$ to be small enough so that it is everywhere transverse to the Reeb vector field $X$, and the path-connected components $\ls\subset\Ws(\gamma)\cap D$ and $\lu\subset\Wu(\gamma)\cap D$ containing $z_0$ intersect only in  $z_0$ and both separate $D$. We write $\lu\setminus\{z_0\}$ and $\ls\setminus\{z_0\}$ as the disjoint union of separatrices $\lu\setminus\{z_0\}=\lu_1\cup\lu_2$ and $\ls\setminus\{z_0\}=\ls_1\cup\ls_2$. 

Let $U\subset D$ be an open neighborhood of $z_0$ that is small enough so that the first-return map $\psi:U\to D$, $\psi(z)=\phi_{\tau(z)}(z)$ is well defined and smooth. Equation~\eqref{e:tangent_space}  implies that $d\psi(z_0)=d\phi_{t_0}(z_0)$, and
\begin{align*}
 T_{z_0}\ls=\Es(z_0),
 \qquad
 T_{z_0}\lu=\Eu(z_0).
\end{align*}
Since $d\psi(z_0)|_{\Es}=\sigma I$ and $d\psi(z_0)|_{\Eu}=\sigma^{-1} I$, and since $\sigma<0$, the first-return map $\psi$ switches the separatrices, i.e.
\begin{align*}
\psi(\ls_i\cap U)\subset\ls_{3-i},
\quad
\psi(\lu_i\cap U)\subset\lu_{3-i},\qquad\forall i=1,2.
\end{align*}
By our assumption, there exist $i,j\in\{1,2\}$ and transverse homoclinic intersections $z\in\lu_i\cap\Ws(\gamma)\cap U$ and $z'\in\ls_j\cap\Wu(\gamma)\cap U$. Therefore, $\psi(z)\in\lu_{3-i}\cap\Ws(\gamma)$ is a transverse homoclinic intersection in the other unstable separatrix, and $\psi(z')\in\ls_{3-j}\cap\Wu(\gamma)$ is a transverse homoclinic intersection in the other stable separatrix. This shows that $\gamma$ has transverse intersections in all the separatrices, which is prevented by Lemma~\ref{l:no_homoclinics_in_all_separatrices} due to the minimality of the broken book decomposition.
\end{proof}

\section{Geodesic flows}
\label{s:geodesic_flows}

We now consider the geodesic flow $\phi_t:SM\to SM$ of a closed Riemannian surface $(M,g)$. Such a $\phi_t$ is the Reeb flow of the Liouville contact form
\begin{align*}
 \lambda_{(x,v)}(w)=g(v,d\pi(x,v)w),\qquad\forall (x,v)\in SM,\ w\in T_{(x,v)}M,
\end{align*}
where $\pi:SM\to M$, $\pi(x,v)=x$ is the base projection. We recall that $(M,g)$ is called \emph{bumpy} when its unit tangent bundle $(SM,\lambda)$ is non-degenerate in the sense of Section~\ref{ss:KS}: none of the Floquet multipliers of the closed orbits of the geodesic flow is a complex root of unity. Moreover, $g$ is said to satisfy the \emph{Kupka-Smale condition} when so does $(SM,\lambda)$ as a closed contact 3-manifold: $(M,g)$ is bumpy and the stable and unstable manifolds of the closed orbits of its geodesic flow intersect transversely.
In this section, we shall provide a simpler proof of  Theorem~\ref{t:Birkhoff_section} in the special case of geodesic flows.

\begin{Thm}
\label{t:Birkhoff_section_geodesic_flow}
On any closed surface, any Riemannian metric satisfying the Kupka-Smale condition admits a Birkhoff section for its geodesic flow.
\end{Thm}

The proof requires the following preliminary lemma. For each orbit $\gamma(t)=\phi_t(x,v)$, we denote by $\overline\gamma(t):=\phi_t(x,-v)$ the orbit associated to the reversed underlying geodesic.

\begin{Lemma}
\label{l:no_intersecting_closed_geodesic}
Let $(M,g)$ be a bumpy closed Riemannian surface,  whose unit tangent bundle is equipped with a minimal broken book decomposition with binding $K=\Krad\cup\Kbr$ and broken binding $\Kbr=\gamma_1\cup...\cup\gamma_n$. Let $x_i:=\pi\circ \gamma_i$ be the closed geodesics underlying the broken binding orbits. Then, there is no closed orbit $\gamma$ of the geodesic flow such that $(\gamma\cup\overline\gamma)\cap K=\varnothing$ and whose underlying geodesic $x:=\pi\circ\gamma$ intersects $x_1\cup...\cup x_n$. 
\end{Lemma}

\begin{proof}
Let us assume by contradiction that there exists a closed  orbit $\gamma$ of the geodesic flow such that $(\gamma\cup\overline\gamma)\cap K=\varnothing$ and whose underlying geodesic $x:=\pi\circ\gamma:\R/T\Z\to M$ intersects some $x_i$. Notice that $x$ and $x_i$ must intersect transversely, since they are distinct closed geodesics. If $x$ preserves the orientation (that is, $TM|_{x}$ is an orientable bundle over the circle), we denote by $\nu$ the vector field defined along $x$ that is normal to $\dot x$, i.e.
\begin{align*}
 \|\nu\|_g\equiv1, \qquad g(\nu(x(t)),\dot x(t))=0,
\end{align*}
and such that $g(\nu(x(t)),\dot x_i(s))>0$ at some  intersection point $x(t)=x_i(s)$; we consider the embedded compact annulus
\begin{align*}
 \Sigma:=\Big\{ (x(t),v)\in SM\ \Big|\ t\in\R/T\Z,\ g(\nu(x(t)),v)\geq0\Big\}
\end{align*}
with boundary $\partial\Sigma=\gamma\cup\overline\gamma$.
If instead $x$ reverses the orientation (that is, $TM|_{x}$ is a non-orientable bundle over the circle), we consider the immersed compact annulus $\Sigma\looparrowright SM$ whose boundary is a double cover $\partial\Sigma\looparrowright\gamma\cup\overline\gamma$, and whose interior is given by
\begin{align*}
 \interior(\Sigma)
 =
 \Big\{ (x(t),v)\ \Big|\ t\in\R/T\Z,\ v\in S_{x(t)}M\setminus\{\dot x(t),-\dot x(t)\}\Big\}.
\end{align*}
In both cases, $\Sigma$ is an immersed surface of section whose boundary $\partial\Sigma$ is disjoint from the binding $K$, and whose interior $\interior(\Sigma)$ is embedded in $SM$ and intersects the broken binding component $\gamma_i=(x_i,\dot x_i)\subset\Kbr$. Therefore, we can apply Lemma~\ref{l:surgery}, which provides a new broken book decomposition of the unit tangent bundle $SM$ with broken binding $\Kbr\setminus\gamma_i$. This contradicts the fact that the original broken book decomposition was minimal.
\end{proof}

\begin{proof}[Proof of Theorem~\ref{t:Birkhoff_section_geodesic_flow}]
Let $(M,g)$ be a closed surface satisfying the Kupka-Smale condition. According to a theorem of Colin, Dehornoy, and Rechtman \cite{Colin:2020tl}, its unit tangent bundle $(SM,\lambda)$ admits a minimal broken book decomposition. All we have to show is that the broken binding $\Kbr$ is empty, so that the broken book is actually a rational open book, and any page is a Birkhoff section for the geodesic flow of $(M,g)$. We prove this by contradiction, assuming that $\Kbr\neq\varnothing$.

Lemma~\ref{l:two_homoclinics} implies that there exists a broken binding component $\gamma\subset\Kbr$ with transverse homoclinics in both stable separatrices. By Lemma~\ref{l:positively_hyperbolic}, $\gamma=(x,\dot x)$ is positively hyperbolic. Let $t_0>0$ be the minimal period of $\gamma$, and $\sigma\in(0,1)$ the stable Floquet multiplier of $\gamma$, i.e.
\begin{align*}
\det(d\phi_{t_0}(\gamma(t))-\sigma I)=0,
\qquad\forall t\in\R. 
\end{align*}
We denote by $\Es$ and $\Eu$ the stable and unstable bundles of $\gamma$, i.e.
\begin{align*}
 \Es(\gamma(t))=\ker(d\phi_{t_0}(\gamma(t))-\sigma I),
 \qquad
  \Eu(\gamma(t))=\ker(d\phi_{t_0}(\gamma(t))-\sigma^{-1} I).
\end{align*}
We consider an open disk $D\subset SM$ containing the point $z_0=\gamma(0)$, with tangent space $T_{z_0}D=\Es(z_0)\oplus\Eu(z_0)$, and small enough so that it is everywhere transverse to the geodesic vector field $X$, and the path-connected components $\ls\subset\Ws(\gamma)\cap D$ and $\lu\subset\Wu(\gamma)\cap D$ containing $z_0$ intersect only at $z_0$ and both separate $D$. We write $\lu\setminus\{z_0\}$ and $\ls\setminus\{z_0\}$ as the disjoint union of separatrices $\lu\setminus\{z_0\}=\lu_1\cup\lu_2$ and $\ls\setminus\{z_0\}=\ls_1\cup\ls_2$. By our assumption on $\gamma$, there exist transverse intersections $z_i\in\ls_i\cap\Wu(\gamma)$ for all $i\in\{1,2\}$. We denote the corresponding homoclinic orbits by $\zeta_i(t):=\phi_t(z_i)$, and the underlying geodesics by $x_i:=\pi\circ\zeta_i$.

We claim that the closed geodesic $x:=\pi\circ\gamma$ underlying $\gamma$ is without conjugate points, that is,
\begin{align*}
d\phi_t(z_0)w\not\in\ker (d\pi(\phi_t(z_0))),
\qquad
\forall t\neq0,\ w\in\ker(d\pi(z_0))\setminus\{0\}.
\end{align*}
Indeed, assume that $x$ has conjugate points.  Under this assumption it is well known that, for each $z\in N\setminus\gamma$ sufficiently close to $z_0$, the corresponding geodesic $y(t):=\pi\circ\phi_t(z)$ intersects  $x$, see e.g.\ \cite[Lemma~5.9]{De-Philippis:2020wz}. 
In particular, the geodesic $x_1$ underlying the homoclinic $\zeta_1$ must intersect $x$ transversely. Since $\zeta_1=(x_1,\dot x_1)$ is a transverse homoclinic of $\gamma$, for each $\epsilon>0$ and  $S>0$, by the shadowing lemma \cite[Theorem~5.3.3]{Fisher:2019vz} there exists a closed orbit $\zeta=(y,\dot y)$ of the geodesic flow such that 
\begin{align*}
\max_{t\in[-S,S]} d(x_1(t),y(t)) < \epsilon.
\end{align*}
Here, $d:M\times M\to[0,\infty)$ is the Riemannian distance. By taking $\epsilon>0$ small enough and $S>0$ large enough, we can ensure that the closed  orbit $\zeta$ and its reverse $\overline\zeta$ are not in the binding $K$, and the underlying closed geodesic $y$ intersects $x$ transversely. This contradicts Lemma~\ref{l:no_intersecting_closed_geodesic}.

Since the closed geodesic $x$ is without conjugate points, the stable bundle $\Es$ intersects trivially the vertical sub-bundle $\ker(d\pi)\subset T(SM)$, i.e.
\begin{align}
\label{e:trivial_intersection_with_vertical}
 \Es(\gamma(t))\cap\ker(d\pi(\gamma(t)))=\{0\},\qquad\forall t\in\R.
\end{align}
Indeed, if $w\in \Es(\gamma(t))\cap \ker(d\pi(\gamma(t)))$, then $d\phi_{t_0}(\gamma(t))w=\sigma w\in\ker(d\pi(\gamma(t)))$, and since $x$ has no conjugate points we must have $w=0$.

We claim that the closed geodesic $x:\R/t_0\Z\to M$ does not reverse the orientation. Indeed, the stable sub-bundle $\Es$ is contained in the contact distribution $\ker(\lambda)$. Therefore
\begin{align*}
0=\lambda(w)=g(\dot x(t),d\pi(\gamma(t)) w),\qquad\forall w\in\Es(\gamma(t)).
\end{align*}
This, together with~\eqref{e:trivial_intersection_with_vertical}, implies that, for each non-zero $w\in\Es(z_0)$, the vector field $W(t):=d(\pi\circ\phi_t)(z_0)w$ is nowhere vanishing, orthogonal to $\dot\gamma(t)$, and such that $W(t_0)=\sigma W(0)$. Since $\sigma>0$, this proves that $x$ does not reverse the orientation, and therefore there exists tubular neighborhood $A\subset M$ of $x$ that is diffeomorphic to an open annulus. We write the  complement of $x$ in this annulus as a union of connected components as \[A\setminus x=A_1\cup A_2.\]

Consider again the homoclinics $\zeta_i(t)=(x_i(t),\dot x_i(t))=\phi_t(z_i)$, for $i=1,2$. For some positive real numbers $a_i,b_i>0$, we have
\begin{align*}
 \lim_{t\to\infty} d\big(\zeta_i(-t-a_i),\gamma(-t)\big)=0,
 \qquad
 \lim_{t\to\infty} d\big(\zeta_i(t+b_i),\gamma(t)\big)=0,
\end{align*}
where $d:SM\to SM\to[0,\infty)$ now denotes the distance on $SM$ induced by $g$.
We already showed that none of the underlying geodesics $x_i$ can intersect $x$. Since $z_1$ and $z_2$ belong to different stable separatrices $\ls_1$ and $\ls_2$, for $t>0$ large enough the points $x_1(t)$ and $x_2(t)$ lie on different sides of the closed geodesic $x$, say $x_1(t)\in A_1$ and $x_2(t)\in A_2$. We have two cases to consider:
\begin{itemize}
\item Assume that, for all $t>0$ large enough, one such homoclinic $x_i$ satisfies $x_i(-t)\in A_{3-i}$. Namely, the homoclinic $x_i$ switches component of $A\setminus x$ as $t$ goes from $-\infty$ to 
$\infty$. For each positive integer $n\in\N$, we set $T_n:=2nt_0+a_i+b_i$ and define the $T_n$-periodic pseudo-orbit $\beta_n:\R\to SM$ by
\begin{align*}
 \beta_n(t+kT_n):=\zeta_i(t),\qquad\forall k\in\Z,\ t\in(-nt_0-a_i,nt_0+b_i].
\end{align*}
For each $n\in\N$ large enough (and therefore $T_n$ large enough) and $\epsilon>0$ small enough, we have that 
$\pi\circ\beta_n(T_n-\epsilon)\in A_i$ and $\pi\circ\beta_n(T_n+\epsilon)\in A_{3-i}$.
The jumps of these pseudo-orbits tend to zero as $n\to \infty$, i.e.
\begin{align*}
\lim_{n\to\infty}\lim_{\epsilon\to 0^+} d(\beta_n(T_n-\epsilon),\beta_n(T_n+\epsilon))=0, 
\end{align*}
since both $\beta_n(T_n-\epsilon)$ and $\beta_n(T_n+\epsilon)$ tend to $\gamma(0)$. 

\item Assume that, for all $t>0$ large enough, we have $x_1(-t)\in A_1$ and $x_2(-t)\in A_2$.  For each positive integer $n\in\N$, we set $T_n:=4nt_0+a_1+b_1+a_2+b_2$ and define the $T_n$-periodic pseudo-orbit $\beta_n:\R\to SM$ by
\begin{align*}
\qquad\qquad
\beta_n(t+kT_n):=
\left\{
  \begin{array}{@{}ll}
    \zeta_1(t+nt_0+b_1), & \forall k\in\Z,\ t\in(-2nt_0-a_1-b_1,0],\vspace{5pt}\\ 
    \zeta_2(t-nt_0-a_2), & \forall k\in\Z,\ t\in(0,2nt_0+a_2+b_2].
  \end{array}
\right.
\end{align*}
For each $n\in\N$ large enough (and therefore $T_n$ large enough) and $\epsilon>0$ small enough, we have $\pi\circ\beta_n(-\epsilon)\in A_1$
and $\pi\circ\beta_n(\epsilon)\in A_2$.
The jumps of these pseudo-orbits tend to zero as $n\to \infty$, i.e.
\begin{gather*}
\lim_{n\to\infty}\lim_{\epsilon\to 0^+}
\Big(
d\big(\beta_n(-\epsilon),\beta_n(\epsilon)\big)
+d\big(\beta_n(T_n-\epsilon),\beta_n(T_n+\epsilon)\big)\Big)=0;
\end{gather*}
since the four points $\beta_n(-\epsilon)$, $\beta_n(\epsilon)$, $\beta_n(T_n-\epsilon)$, $\beta_n(T_n+\epsilon)$ all tend to $\gamma(0)$.
\end{itemize}

In both cases, for any $\epsilon>0$ and for all $n\in\N$ large enough (so that $T_n$ is large enough), the shadowing lemma \cite[Theorem~5.3.3]{Fisher:2019vz} implies that there exists a periodic orbit $\beta=(w,\dot w)$ of the geodesic flow that is $\epsilon$-close to the pseudo-orbit $\beta_n$ up to time-reparametrization. By choosing $\epsilon>0$ small enough and $n\in\N$ large enough, we infer that the closed orbit $\beta$ does not belong to the binding $K$, and the underlying closed geodesic $w=\pi\circ\beta$ intersects $x$ transversely. However, this contradicts Lemma~\ref{l:no_intersecting_closed_geodesic}.
\end{proof}

\section{Reeb flows}
\label{s:Reeb}

The following theorem, which may have independent interest, is the last ingredient for the proof of Theorem~\ref{t:Birkhoff_section}.
In the statement, we employ the terminology introduced in Section~\ref{ss:transverse_homoclinics_in_all_separatrices}.

\begin{Thm}
\label{t:homoclinics}
Let $(N,\lambda)$ be a closed contact 3-manifold satisfying the Kupka-Smale condition,
equipped with a broken book decomposition with broken binding $\Kbr$.
Any component $\gamma\subset \Kbr$ has homoclinics in all the separatrices, and satisfies 
\[\overline{\Ws(\gamma)}=\overline{\Wu(\gamma)}.\]
\end{Thm}

Let us first employ this theorem in order to conclude the proof of Theorem~\ref{t:Birkhoff_section}.

\begin{proof}[Proof of Theorem~\ref{t:Birkhoff_section}]
Let $(N,\lambda)$ be a closed contact 3-manifold satisfying the Kupka-Smale condition. We consider a minimal broken book decomposition of $(N,\lambda)$, with binding $K=\Krad\cup\Kbr$ and pages $\FF$. Lemma~\ref{l:no_homoclinics_in_all_separatrices} implies that no broken binding component $\gamma\subset\Kbr$ has  homoclinics in all the separatrices. Therefore, Theorem~\ref{t:homoclinics} implies that the broken binding $\Kbr$ is empty. Namely, the minimal broken book is a rational open book, and any page $\Sigma\in\FF$ is a Birkhoff section.
\end{proof}

The proof of Theorem~\ref{t:homoclinics} will require three preliminary lemmas. From now on, we consider a closed contact manifold $(N,\lambda)$ satisfying the Kupka-Smale condition, equipped with a broken book decomposition with binding $K=\Krad\cup\Kbr$ and pages $\FF$. As usual, we denote by $X$ the Reeb vector field, and by $\phi_t:N\to N$ the Reeb flow.

\begin{Lemma}
\label{l:large_area}
Let $\Sigma\subset N$ be a surface of section, and $\gamma$ a hyperbolic closed Reeb orbit of minimal period $p>0$. Let $W$ be a path-connected component of either $\Ws(\gamma)\setminus\gamma$ or $\Wu(\gamma)\setminus\gamma$. Any  embedded disk $D\subset\Sigma$ whose boundary $\partial D$ is a smooth embedded circle in $\Sigma\cap W$ has area
\begin{align*}
 \area(D,d\lambda):=\int_{D} d\lambda \geq p.
\end{align*}
Here, $D$ is oriented so that $d\lambda|_D$ is a positive area form.
\end{Lemma}

\begin{proof}
The Reeb vector field is tangent to $W$. 
 Since the boundary circle $\partial D$ is contained in $\Sigma$, 
 it is transverse to the Reeb vector field $X$. If $\partial D$ were contractible in $W$, it would bound a disk $D'\subset W$, and Poincar\'e-Bendixon theorem would imply that $D'$ contains a closed orbit of the Reeb flow; this would contradict the fact that all points in $W$ are asymptotic to $\gamma$ (in the future or in the past, depending on whether $W\subset\Ws(\gamma)$ or $W\subset\Wu(\gamma)$). We conclude that $\partial D$ is non-contractible in $W$. 
 Since $W$ is homeomorphic to an open annulus, there exists an open sub-annulus $A\subset W$ with boundary $\partial A=\gamma\cup\partial D$. If $\gamma$ is positively hyperbolic, then $\overline A$ is embedded in $N$, and we set $m:=1$. If instead $\gamma$ is negatively hyperbolic, then $\partial A\setminus\partial D$ is a double cover of the closed Reeb orbit $\gamma$, and we set $m:=2$. 
We orient $D$ by means of the area form $d\lambda|_D$, and $\partial D$ as its boundary. Since $d\lambda(X,\cdot)\equiv 0$, and since $A$ is tangent to $X$, we conclude that
\[
\area(D,d\lambda)
=\int_D d\lambda 
=\int_{\partial D} \lambda
=\int_{A} d\lambda + m\int_\gamma \lambda = m\,p\geq p.
\qedhere
\]
\end{proof}

\begin{Lemma}\label{l:closures_Ws_Wu}
Any broken binding component $\gamma\subset\Kbr$ having a homoclinic satisfies
\[\overline{\Wu(\gamma)}=\overline{\Ws(\gamma)}.\]
\end{Lemma}

\begin{figure}
\begin{footnotesize}
\begingroup%
  \makeatletter%
  \providecommand\color[2][]{%
    \errmessage{(Inkscape) Color is used for the text in Inkscape, but the package 'color.sty' is not loaded}%
    \renewcommand\color[2][]{}%
  }%
  \providecommand\transparent[1]{%
    \errmessage{(Inkscape) Transparency is used (non-zero) for the text in Inkscape, but the package 'transparent.sty' is not loaded}%
    \renewcommand\transparent[1]{}%
  }%
  \providecommand\rotatebox[2]{#2}%
  \newcommand*\fsize{\dimexpr\f@size pt\relax}%
  \newcommand*\lineheight[1]{\fontsize{\fsize}{#1\fsize}\selectfont}%
  \ifx\svgwidth\undefined%
    \setlength{\unitlength}{186.32245122bp}%
    \ifx\svgscale\undefined%
      \relax%
    \else%
      \setlength{\unitlength}{\unitlength * \real{\svgscale}}%
    \fi%
  \else%
    \setlength{\unitlength}{\svgwidth}%
  \fi%
  \global\let\svgwidth\undefined%
  \global\let\svgscale\undefined%
  \makeatother%
  \begin{picture}(1,0.60648685)%
    \lineheight{1}%
    \setlength\tabcolsep{0pt}%
    \put(0,0){\includegraphics[width=\unitlength,page=1]{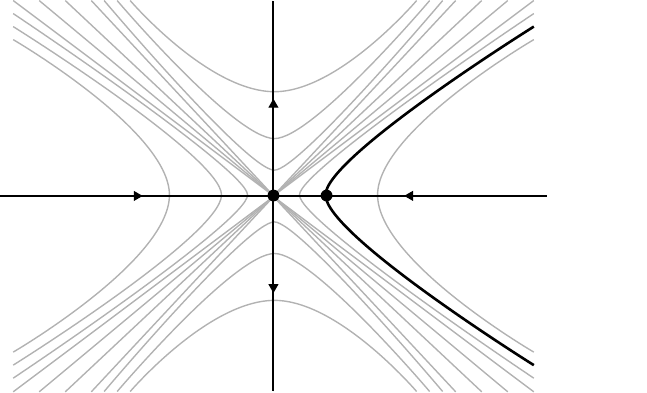}}%
    \put(0.73257007,0.26347949){\color[rgb]{0,0,0}\makebox(0,0)[lt]{\lineheight{1.25}\smash{\begin{tabular}[t]{l}$\Ws(\gamma)$\end{tabular}}}}%
    \put(0.78087277,0.49694572){\color[rgb]{0,0,0}\makebox(0,0)[lt]{\lineheight{1.25}\smash{\begin{tabular}[t]{l}$\Sigma$\end{tabular}}}}%
    \put(0.52325543,0.31329762){\color[rgb]{0,0,0}\makebox(0,0)[lt]{\lineheight{1.25}\smash{\begin{tabular}[t]{l}$S$\end{tabular}}}}%
    \put(0.44274987,0.32939864){\color[rgb]{0,0,0}\makebox(0,0)[lt]{\lineheight{1.25}\smash{\begin{tabular}[t]{l}$\gamma$\end{tabular}}}}%
    \put(0.43469954,0.57745127){\color[rgb]{0,0,0}\makebox(0,0)[lt]{\lineheight{1.25}\smash{\begin{tabular}[t]{l}$\Wu(\gamma)$\end{tabular}}}}%
  \end{picture}%
\endgroup%

\end{footnotesize}
\caption{View in a cross section: the page $\Sigma$ intersecting the path-connected component $L\subset\Ws(\gamma)\setminus\gamma$ in an embedded circle $S$ near the broken binding component~$\gamma$.}
\label{f:circle}
\end{figure}
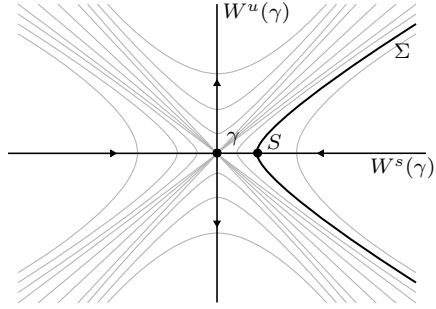

\begin{proof}
We shall only prove that 
\begin{equation}\label{e:Ws_gamma_sub_ov_Wu_gamma}
\Ws(\gamma)\subset \overline{\Wu(\gamma)}.
\end{equation}
The other inclusion $\Wu(\gamma)\subset\overline{\Ws(\gamma)}$ 
follows by applying~\eqref{e:Ws_gamma_sub_ov_Wu_gamma} 
 to the Reeb vector field $-X$ of the contact form $-\lambda$.

Let $L\subset\Ws(\gamma)\setminus\gamma$ be a path-connected component containing homoclinics, i.e.
\begin{align*}
L\cap\Wu(\gamma)\neq\varnothing.
\end{align*}
By the $\lambda$-lemma \cite[Prop.~6.1.10]{Fisher:2019vz}, $L$ accumulates on $\Ws(\gamma)$, i.e.\ $\Ws(\gamma)\subset \overline{L}$.
Therefore, in order to prove~\eqref{e:Ws_gamma_sub_ov_Wu_gamma}, it is enough to show that 
\begin{align}
\label{e:L_sub_ov_Wu_gamma}
L\subset \overline{\Wu(\gamma)}. 
\end{align}
We choose a page $\Sigma\in\F$ of the broken book decomposition 
whose intersection with $L$ contains an embedded circle
$S\subset \Sigma\cap L$ that is non-contractible in $L$ and satisfies 
\begin{gather}\label{e:forward_flow_cap_Sigma_empty}
\phi_t(S)\cap \Sigma = \varnothing, \qquad \forall t>0, 
\end{gather}
see Figure~\ref{f:circle}.
Since $S$ is non-contractible in $L$ and transverse to the Reeb vector field $X$, we have
\[
L=\textstyle\bigcup_{t\in\R}\phi_t(S).
\]
Since $\Wu(\gamma)$ is invariant by the Reeb vector field $X$, in order to prove~\eqref{e:L_sub_ov_Wu_gamma} it is enough to show that 
\[S\subset\overline{\Wu(\gamma)}.\]
We will argue by contradiction, assuming that there exists a non-empty path-connected component
\begin{align}
\label{e:ell_subset_Wuc}
 \ell\subset S\setminus\overline{\Wu(\gamma)}.
\end{align}

Since $L$ contains homoclinics, there exists an open interval $\ell_0\subset L$ such that $\ell\subset\ell_0$ and $\partial\ell_0\subset\Wu(\gamma)$.
Since $(N,\lambda)$ satisfies the Kupka-Smale condition, the intersection $\Wu(\gamma)\cap L$ is transverse. This, together with the $\lambda$-lemma \cite[Prop.~6.1.10]{Fisher:2019vz}, implies that the stable manifold $\Ws(\gamma)$ and the unstable manifold $\Wu(\gamma)$ accumulate on themselves in the $C^1$ topology.
In particular, $\Ws(\gamma)$ accumulates on $\ell_0$ in the $C^1$ topology. Therefore, there exists a heteroclinic rectangle $R\subset\interior(\Sigma)$ as in Figure~\ref{f:heteroclinic_rectangle}: $R$ is an open disk with boundary $\partial R=\ell_0\cup\ell_1\cup q_0\cup q_1$, with $\ell_1\subset\Ws(\gamma)$ and $q_0,q_1\subset\Wu(\gamma)$. Moreover, we can choose such a heteroclinic rectangle to be arbitrarily thin (by choosing $\ell_1$ to be $C^1$-close to $\ell_0$), and in particular so that
\begin{align}
\label{e:R_small_area}
 \area(R,d\lambda):=\int_R d\lambda < p,
\end{align}
where $p>0$ is the minimum among the periods of the broken binding orbits, i.e.
\begin{align}
\label{e:minimal_period_broken_binding}
p:=\min\big\{t>0\ \big|\ \Fix(\phi_t)\cap\Kbr\neq\varnothing\big\}.
\end{align}

\begin{figure}
\begin{footnotesize}
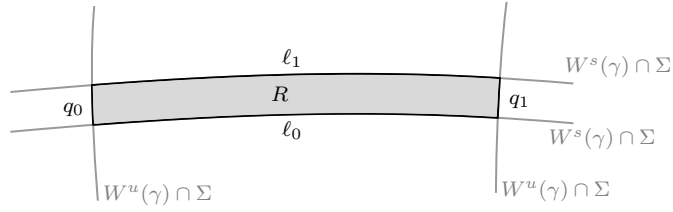
\end{footnotesize}
\caption{The heteroclinic rectangle $R$.}
\label{f:heteroclinic_rectangle}
\end{figure}

For each integer $n\geq0$, we denote by $\tau_n:\Sigma\to(0,\infty]$ the $n$-th return time to the page $\Sigma$. Such functions are defined as $\tau_0\equiv0$ and, for increasing values of $n>0$, 
\begin{align}
\label{e:return_time}
\tau_{n}(z):=\inf\big\{t>\tau_{n-1}(z)\ \big|\ \phi_t(z)\in\Sigma \big\}.
\end{align}
 Analogously, we denote by $\tau_{-n}:\Sigma\to[-\infty,0)$ the $-n$-th return time to the page $\Sigma$, which is defined for increasing values of $n>0$ by
\begin{align}
\label{e:back_return_time}
\tau_{-n}(z):=\sup\big\{t<\tau_{-n+1}(z)\ \big|\ \phi_t(z)\in\Sigma \big\}.
\end{align}
In~\eqref{e:return_time} and~\eqref{e:back_return_time},  we adopt the usual conventions $\inf\varnothing=\infty$ and $\sup\varnothing=-\infty$. For each $n\in\Z$, we denote by $U_n\subset \Sigma$ the open subset over which $\tau_n$ is finite. On every such subset $U_n$, we have a well-defined $n$-th return map
\begin{align*}
 \psi_n:U_n\to U_{-n},
 \qquad
 \psi_n(z)=\phi_{\tau_n(z)}(z),
\end{align*}
which is a diffeomorphism preserving the area form $d\lambda$, i.e.
\begin{align*}
 \psi_n^*(d\lambda|_{U_{-n}})=d\lambda|_{U_n}.
\end{align*}
Notice that $\psi_0=\mathrm{id}$ and $\psi_n^{-1}=\psi_{-n}$.

By~\eqref{e:forward_flow_cap_Sigma_empty}, the forward flowout $\phi_t(\ell_0)$, for $t>0$, does not intersect the page $\Sigma$. 
Since $\ell_1\subset \Ws(\gamma)$, there exists some positive integer $m>0$ such that, for each $z\in\ell_1$, 
the forward orbit $\phi_t(z)$ intersects $\Sigma$ for at most $m$ positive values of $t$. Therefore
\begin{align*}
U_{m+1}\cap(\ell_0\cup\ell_1)=\varnothing.
\end{align*}
Let $A\subset R\setminus\overline{\Wu(\gamma)}$ be the connected component whose boundary contains the open interval $\ell\subset\ell_0$. By Poincar\'e recurrence, 
there exists an arbitrarily large positive integer $n>m$ and a point $z_0\in A\cap U_{n}$ such that $\psi_{n}(z_0)\in A$. 
Let $B\subseteq A\cap U_{n}$ be the connected component containing $z_0$.

We claim that 
\begin{equation}\label{eclaim1}
\partial(\psi_{n}(B))\setminus \overline{\Wu(\gamma)}\subset
\Sigma\setminus U_{-n} 
= 
\big\{\tau_{-n}=-\infty\big\}.
\end{equation}
Indeed, if the claim is false, there exists a point 
$y\in U_{-n}\cap\partial(\psi_n(B))\setminus\overline{\Wu(\gamma)}$.
This implies that $\psi_{-n}(y)\in U_{n}\cap \partial B\setminus \overline{\Wu(\gamma)}$. 
Since $\partial B\subset\partial A\cup \partial U_n$, $\partial A\subset \ell_0\cup \ell_1\cup \overline{\Wu(\gamma)}$, and $n>m$, we obtain the contradiction 
\begin{align*}
\psi_{-n}(y)\in U_n\cap\partial A\setminus \overline{\Wu(\gamma)}
\subset
U_n\cap(\ell_0\cup\ell_1)
\subseteq
U_{m+1}\cap(\ell_0\cup\ell_1)
=\varnothing.
\end{align*}

Next, we claim that
\begin{align}
\label{e:hard}
\partial(\psi_n(B))\subset\overline{\Wu(\gamma)}.
\end{align}
Indeed, assume that there exists $y\in\partial(\psi_i(B))\setminus\overline{\Wu(\gamma)}$. We choose a sequence $y_k\in \psi_n(B)$ such that $y_k\to y$, and a corresponding sequence $w_k\in\partial(\psi_n(B))$ such that 
\begin{align*}
d(y_k,w_k)=\min_{w\in\partial(\psi_n(B))} d(y_k,w)=:r_k. 
\end{align*}
Here, $d:\Sigma\times\Sigma\to[0,\infty)$ denotes the distance induced by a fixed auxiliary Riemannian metric on $\Sigma$. The open Riemannian disk 
\[D_k:=\{w\in\Sigma\ |\ d(y_k,w)< r_k\}\] 
is contained in $\psi_n(B)$. Since $y_k\to y$ and $d(y_k,w_k)\leq d(y_k,y)$, we infer that $w_k\to y$ as well. Since $y\not\in\overline{\Wu(\gamma)}$, we infer that $w_k\not\in\overline{\Wu(\gamma)}$ for all $k$ large enough. We fix one such $w_k$, and consider the radial geodesic
\begin{align*}
 \zeta_0:[0,1]\to \overline{D_k},\qquad\zeta_0(s)=\exp_{y_k}(s  \exp_{y_k}^{-1}(w_k)),
\end{align*}
which joins $\zeta_0(0)=y_k$ and $\zeta_0(1)=w_k$. We define the smooth function 
\begin{align*}
\tau:[0,1)\to(-\infty,0),
\qquad
\tau(s)=\tau_{-n}(\zeta_0(s)), 
\end{align*}
so that $\psi_{-n}(\zeta_0(s))=\phi_{\tau(s)}(\zeta_0(s))$. Since $\zeta_0(1)\not\in U_{-n}=\{\tau_{-n}>-\infty\}$, the function $\tau$ is unbounded. The backward orbit $\phi_{-t}(\zeta_0(1))$ intersects the page $\Sigma$ for at most $n-1$ values of $t>0$. In particular, $\zeta_0(1)$ belongs to the unstable manifold of the broken binding $\Wu(\Kbr)$. If $\ell'\subset \Wu(\Kbr)\cap\Sigma$ is a sufficiently small open neighborhood of $\zeta_0(1)$, for each point $w\in \ell'$ the backward orbit $\phi_{-t}(w)$ intersects the page $\Sigma$ for at most $n-1$ values of $t>0$. This, together with the fact that $D_k\subset \psi_n(B)\subset U_{-n}$,
 implies that $\ell'\cap D_k=\varnothing$, and therefore $\ell'$ has tangent space $T_{\zeta_0(1)}\ell'=T_{\zeta_0(1)}\partial D_k$. Since $\zeta_0$ is a radial geodesic, it intersects the boundary $\partial D_k$ transversely, and therefore we infer the transversality condition 
\[\dot\zeta_0(1)\pitchfork \Wu(\Kbr).\]
This, together with~\cite[Prop.~2.2]{Contreras:2021tg}, implies that the path 
\[\zeta_1=\psi_{-n}\circ\zeta_0:[0,1)\to B\] accumulates on some embedded circle $S'\subset\Ws(\Kbr)\cap\Sigma$, i.e.
\begin{align*}
S'\subset\overline{\zeta_1([0,1))}.
\end{align*}
In particular, $S'$ is contained in the rectangle $\overline R$, and therefore there exists an open disk $D'\subset R$ with boundary $\partial D'=S'$. This, together Lemma~\ref{l:large_area}, implies
\begin{align}
\label{e:area_low_bound_proof}
 \area(R,d\lambda)\geq \area(D',d\lambda) \geq p,
\end{align}
where $p$ is the minimum among the periods of the broken binding orbits, as defined in~\eqref{e:minimal_period_broken_binding}. The lower bound~\eqref{e:area_low_bound_proof} contradicts~\eqref{e:R_small_area}. This concludes the proof of~\eqref{e:hard}.

Summing up, the connected component $B\subset U_n\cap A$ satisfies $\psi_n(B)\cap A\neq\varnothing$ and $\partial(\psi_n(B))\cap A=\varnothing$. Therefore
\begin{align}\label{e:A_psin}
B\subseteq A\subseteq \psi_n(B).
\end{align}
Since $\psi_n$ preserves the area form $d\lambda|_\Sigma$, we conclude
\begin{align*}
 \area(A) = \area(\psi_n(B)).
\end{align*}
We claim that
 \begin{equation}\label{e:ell}
\ell\subset \partial(\psi_n(B)). 
 \end{equation}
Indeed, \eqref{e:A_psin} implies that $\overline{A}\subset\overline{\psi_n(B)}$, 
and therefore
\begin{equation}\label{e:ellapsi}
\ell\subset\partial A\subset\psi_n(B)\cup\partial(\psi_n(B)).
\end{equation}
Suppose that $\ell\cap \psi_n(B)\ne\varnothing$.
Since $\psi_n(B)$ is open and $\ell\subset \partial A$, the open set
$\psi_n(B)\setminus \overline A$ is non-empty. Therefore, we get the following contradiction
\[
\area(\psi_n(B))
=
\area(A) + \area(\psi_n(B)\setminus\overline A)
>
\area(A)
=
\area(\psi_n(B))
.
\] 
We conclude that $\ell\cap \psi_n(B)=\varnothing$, which together with 
\eqref{e:ellapsi} proves \eqref{e:ell}.

All together, \eqref{e:ell} and \eqref{e:hard}
provide the inclusions 
\[
\ell\subset \partial(\psi_n(B))\subset \overline{\Wu(\gamma)},
\]
which contradict~\eqref{e:ell_subset_Wuc}.
\end{proof}

\begin{Lemma}\label{l:homoclinics_heteroclinics}
Let $\alpha,\beta\subset\Kbr$ be broken binding components, and assume that both have homoclinics. 
\begin{enumerate}

\item[$(i)$] If a path-connected component $P\subset\Wu(\alpha)\setminus\alpha$ satisfies $P\cap \Ws(\beta)\ne\varnothing$, then
$\Ws(\alpha)\cap \Wu(\beta)\ne\varnothing$ and $P\cap \Ws(\alpha)\ne \varnothing$.

\item[$(ii)$] If a path-connected component $Q\subset\Ws(\beta)\setminus\beta$ satisfies $Q\cap \Wu(\alpha)\ne \varnothing$, then $\Ws(\alpha)\cap \Wu(\beta)\ne \varnothing$ and $Q\cap \Wu(\beta)\ne \varnothing$.

\end{enumerate}
\end{Lemma}

\begin{proof}
Point (ii) is obtained by applying point (i) to the Reeb vector field $-X$
corresponding to the contact form $-\lambda$. Therefore, we only need to prove point (i).
The statement is tautological if $\alpha=\beta$, so we assume that $\alpha\ne\beta$.

We fix a point $z_0\in\beta$, and an embedded open disk $D\subset N$ transverse to $X$ and containing the point $z_0$. We denote by $\lu\subset\Wu(\beta)\cap D$ and $\ls\subset\Ws(\beta)\cap D$ the path-connected components containing $z_0$, and require $D$ to be small enough so that 
$\alpha\cap D=\varnothing$, and $\lu$ and $\ls$ are embedded 1-dimensional manifolds intersecting only at $z_0$, both separating $D$ into two path-connected components.

By assumption, $\beta$ has a homoclinic, which must be a transverse homoclinic since $(N,\lambda)$ satisfies the Kupka-Smale condition. This, together with the $\lambda$-lemma \cite[Prop.~6.1.10]{Fisher:2019vz}, implies that the stable and unstable manifolds $\Ws(\beta)$ and $\Wu(\beta)$ accumulate on themselves in the $C^1$ topology. In particular, there exist sequences of path-connected components $\lu_n\subset\Wu(\beta)\cap D$ and $\ls_n\subset\Ws(\beta)\cap D$ such that $\lu_n\to\lu$ and $\ls_n\to\ls$ in the $C^1$-topology, thus forming a grid as in Figure~\ref{f:grid}. We set
\begin{align*}
 L:=\bigcup_{n} \,(\lu_n\cup\ls_n),
\end{align*}
and denote by $p>0$ the minimal period of the closed Reeb orbit $\alpha$. We consider a connected component $R_0\subset D\setminus L$ which is surrounded by eight connected components $R_1,...,R_8\subset D\setminus L$, every such connected component $R_j$ being a heteroclinic rectangle: an open disk whose piecewise smooth boundary is the union of two compact segments in $\Wu(\beta)\cap D$ and two compact segments in $\Ws(\beta)\cap D$. We can find such $R_j$'s so that the whole heteroclinic rectangle
\begin{align*}
 R:=\interior( \overline{R_0\cup R_1\cup...\cup R_8} )
\end{align*}
has arbitrarily small area, and in particular so that
\begin{align}
\label{e:R_small_area_last_proof}
 \area(R,d\lambda):=\int_R d\lambda < p.
\end{align}

\begin{figure}
\begin{footnotesize}
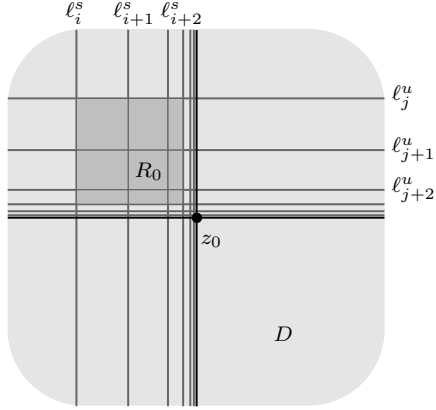
\end{footnotesize}
\caption{The grid in the disk $D$ formed by the stable and unstable manifolds of $\beta$.}
\label{f:grid}
\end{figure}

By Lemma~\ref{l:closures_Ws_Wu}, we have 
\begin{align}
\label{e:oWu=oWs}
\overline{\Wu(\alpha)}=\overline{\Ws(\alpha)}. 
\end{align}
Since there are heteroclinic intersections $P\cap\Ws(\beta)\neq\varnothing$, which are transverse according to the Kupka-Smale condition, the $\lambda$-lemma implies that 
\[\Wu(\beta)\subset\overline{\Wu(\alpha)}.\] 
In particular there exists a point $z\in\partial R_0\cap \overline{\Wu(\alpha)}$, and therefore a point 
\[z'\in R\cap \Ws(\alpha)\] 
according to \eqref{e:oWu=oWs}. Let $\ell'\subset\overline R\cap \Ws(\alpha)$ be the path-connected component containing $z'$, which is an injectively immersed 1-dimensional manifold. Such an $\ell'$ cannot be a circle: if this were the case, it would bound a disk $D'\subset \overline R$; by Lemma~\ref{l:large_area}, we would have $\area(R,d\lambda)\geq\area(D',\lambda)>p$, contradicting~\eqref{e:R_small_area_last_proof}. Therefore $\ell'$ is an injectively immersed interval. Since $D\cap\alpha=\varnothing$, we infer that $\ell'$ cannot be entirely contained in the open disk $R$. Therefore $\ell'\cap\partial R\neq\varnothing$. Notice that $\partial R\subset\Wu(\beta)\cap\Ws(\beta)$. Since the stable manifolds $\Ws(\alpha)$ and $\Ws(\beta)$ are disjoint, we must have $\ell'\cap\partial R\subset\Wu(\beta)$, which proves that
\begin{align}
\label{e:heteroclinic_back}
\Ws(\alpha)\cap\Wu(\beta)\neq\varnothing.
\end{align}
Once again, this intersection is transverse due to the Kupka-Smale condition.

Finally, since there are transverse heteroclinic intersections $P\cap\Ws(\beta)\neq\varnothing$, the $\lambda$-lemma implies that $P$ accumulates on $\Wu(\beta)$ in the $C^1$ topology. This, together with the non-trivial transverse intersection~\eqref{e:heteroclinic_back}, implies that 
\[
\Ws(\alpha)\cap P\neq\varnothing.
\qedhere
\]
\end{proof}

\begin{proof}[Proof of Theorem~\ref{t:homoclinics}]
We denote by $\Gamma$ the family of sequences of arbitrary length \[(\kappa_1,...,\kappa_n),\] where $\kappa_1,...,\kappa_n\subset\Kbr$ are broken binding components such that 
\[\Wu(\kappa_i)\cap \Ws(\kappa_{i+1})\setminus(\kappa_i\cup\kappa_{i+1})\ne\varnothing,\qquad \forall i=1,...,n-1.\] 
This definition readily implies that 
\begin{equation*}
(\kappa_1,\kappa_2,\kappa_3)\in \Gamma,\qquad\forall
(\kappa_1,\kappa_2), (\kappa_2,\kappa_3)\in \Gamma.
\end{equation*}
By the $\lambda$-lemma \cite[Prop.~6.1.10]{Fisher:2019vz}, we have
\begin{gather*}
(\kappa_1,\kappa_n)\in\Gamma,
\qquad
\forall (\kappa_1,\ldots,\kappa_n)\in\Gamma.
\end{gather*}
By Lemma~\ref{l:heteroclinics}, for any $\gamma\in\Kbr$ there are $\alpha,\,\beta\in \Kbr$ such that $(\alpha,\gamma)\in\Gamma$ and
$(\gamma,\beta)\in\Gamma$. 

Now, consider an arbitrary broken binding component $\gamma\subset\Kbr$. The above properties imply that there exists an infinite sequence $(\gamma,\kappa_1,\kappa_2,\ldots)\in\Gamma$ 
starting at $\gamma$. Since $\Kbr$ is finite, there are $n< m$ such that $\kappa_n=\kappa_m=:\beta$, and
therefore $(\gamma,\beta,\beta)\in \Gamma$.
An analogous argument implies that $(\alpha,\alpha,\gamma)\in\Gamma$ for some broken binding component $\alpha\subset\Kbr$, and all together we obtain
\begin{align}
\label{e:aagbb}
 (\alpha,\alpha,\gamma,\beta,\beta)\in\Gamma.
\end{align}
In particular $(\alpha,\beta)\in\Gamma$ and both $\alpha$ and $\beta$ have homoclinics. Lemma~\ref{l:homoclinics_heteroclinics}(i) implies that $(\beta,\alpha)\in\Gamma$. This latter sequence and the one in~\eqref{e:aagbb} imply that $(\gamma,\beta,\alpha,\gamma)\in\Gamma$, and hence \[(\gamma,\gamma)\in\Gamma.\]
This proves that every broken binding orbit has a homoclinic.

It remains to prove that every broken binding component $\alpha\subset\Kbr$ has homoclinics in all the separatrices. Consider an arbitrary path-connected component $P\subset\Wu(\alpha)\setminus\alpha$. By Lemma~\ref{l:heteroclinics}(i), there exists a broken binding component $\beta\subset\Kbr$ such that $P\cap \Ws(\beta)\ne\varnothing$. Since we already know that both $\alpha$ and $\beta$ have homoclinics, we can apply  Lemma~\ref{l:homoclinics_heteroclinics}(i), which gives  $P\cap \Ws(\alpha)\ne \varnothing$, that is, $P$ contains homoclinics. Analogously, using Lemma~\ref{l:heteroclinics}(ii) and Lemma~\ref{l:homoclinics_heteroclinics}(ii), any path-connected component $Q\subset\Ws(\alpha)\setminus\alpha$ contains homoclinics. Finally, the equality $\overline{\Ws(\alpha)}=\overline{\Wu(\alpha)}$ was already established in Lemma~\ref{l:closures_Ws_Wu}.
\end{proof}

\bibliographystyle{amsalpha}
\bibliography{biblio}

\end{document}